\documentclass[11pt]{article}

\usepackage[top=50mm, bottom=50mm, left=50mm, right=50mm]{geometry}
\usepackage{lineno}
\usepackage{amssymb}
\usepackage{amsmath}
\usepackage{amsthm}
\usepackage{epsfig}
\usepackage{graphicx}
\usepackage{graphics}
\usepackage{float}
\usepackage{subfigure}
\usepackage{multirow}
\usepackage{color}
\usepackage{lineno}
\usepackage{fullpage}
\usepackage[normalem]{ulem} 
\usepackage{makeidx}
\usepackage{xspace}
\usepackage{wrapfig}
\usepackage{hyperref}
\usepackage{tikz}
\usetikzlibrary{arrows.meta,positioning,calc, decorations.pathmorphing, decorations.pathreplacing, decorations.shapes}
\makeindex

\newtheorem{theorem}{Theorem}

\newtheorem{corollary}{Corollary}

\newtheorem{Assumption}{Assumption}
\newtheorem{definition}{Definition}

\newtheorem{lemma}{Lemma}

\definecolor{darkred}{rgb}{1, 0.1, 0.3}
\definecolor{darkblue}{rgb}{0.1, 0.1, 1}
\definecolor{darkgreen}{rgb}{0,0.6,0.5}

\newcommand{\cK}{\mathcal{K}}

\newcommand{\jumpS}{\hat{S}}
\newcommand{\jumpI}{\hat{I}}
\newcommand{\jumpA}{\hat{A}}
\newcommand{\probg}{\mathbb{P}^{G_0}}
\newcommand{\prob}{\mathbb{P}}
\newcommand{\E}{\mathbb{E}}

\begin{document}

\title{The effects of individual versus community-influenced isolation on SIS epidemic persistence on finite random graphs}
 
\author{
{Shirshendu Chatterjee\footnote{Department of Mathematics, City University of New York, Graduate Center and CCNY}} \and {David Sivakoff\footnote{Departments of Statistics and Mathematics, The Ohio State University}} \and {Matthew Wascher\footnote{Department of Mathematics, Applied Mathematics, and Statistics, Case Western Reserve University}}
}
\maketitle

\setcounter{page}{0}

\begin{abstract}
The contact process, or the SIS epidemic model, is a continuous-time Markov process used to model the spread of a recurring infection on a graph. Each vertex is either healthy or infected, and each infected vertex independently infects each of its healthy neighbors at rate $\lambda$ and recovers at rate $1$. We study the contact process in the presence of additional intervention measures by introducing a third possible state for vertices, which we call isolated. Vertices may enter the isolated state either because of individual decisions or due to community-influenced decisions, which leads to two distinct models that we call the isolation model and the vigilance model, respectively. In the isolation model, infected vertices self-isolate at rate $\alpha$. In the vigilance model, each healthy vertex causes each of its infected neighbors to isolate at rate $\alpha$. Unlike the classical contact process, these models lack the key features of attractiveness and existence of a dual, which makes analyzing them more challenging. We study the persistence times of the infection on large, finite, degree-heterogeneous random graphs. We show that the infection in the isolation model persists for at least stretched exponential time in the size of the graph for all values of $\alpha$ and $\lambda$. By contrast, in the vigilance model, for every fixed $\alpha$ the persistence time of the infection exhibits a phase transition in $\lambda$: for small $\lambda$ the infection persists for at most a linear time in the size of the graph, while for large $\lambda$ the infection persists exponentially long. This contrast demonstrates that individual versus community-influenced isolation can substantially affect the persistence of an epidemic.

As a corollary of our main results, we also prove prolonged persistence for the SIRS epidemic model on large finite inhomogeneous random graphs, disproving predictions based on non-rigorous analysis published in the physics literature. In the SIRS model,  each vertex is either (i) healthy and susceptible, or (ii) infected, or (iii) removed and temporarily immune. Each infected vertex independently infects each of its healthy neighbors at rate $\lambda$ and recovers with temporary immunity at rate $\alpha$, where this immunity is lost at rate 1. We show that for any $\lambda, \alpha>0$ the persistence time of the SIRS model on the configuration model having $n$ vertices starting from all vertices infected is at least $\exp(n^{1-\eta})$ with high probability for any $\eta>0$.
\end{abstract}

\section{Introduction}

The \textit{contact process}, introduced by \cite{Harris1974}, is one of the popular models in the context of epidemic spread and it has long been studied as a stochastic epidemic model for recurring infections spreading on graphs. Given a graph $G = (V,E)$ with vertices $V$ and edges $E$, the \textit{classical contact process} allows for two possible states on each vertex, healthy and infected, and has a single parameter $\lambda$ that determines the rate at which infection spreads across edges from infected to healthy vertices. The contact process has been studied extensively on lattices \cite{liggett2013stochastic}, trees \cite{pemantle1992contact}, and general random graphs \cite{bhamidi2021survival}. 

Recently, there has been increased interest in studying different variants and improvisations of the contact process models that include additional vertex states or edge behaviors as a way to model more complex pathogen dynamics and human behavior. We consider two such models here. The first, the \textit{contact process with isolation}, includes an additional isolated state in which vertices can neither transmit nor receive infection, and has infected vertices self-isolate at rate $\alpha$. The second, the \textit{contact process with vigilance}, also includes the isolated state but instead of self-isolation, each healthy vertex isolates each of its infected neighbors at rate $\alpha$. These models are intended to model different ways in which individuals or communities might respond to the presence of infection. An important point of comparison between these models is the survival time of the infection on finite graphs with $n$ vertices. We ask whether there exist distinct parameter regimes (phases) where the survival time is short (at most polynomial in $n$) versus long (at least super-polynomial in $n$).

An important difference between our models and the classical contact process is that the classical contact process has a type of monotonicity called \textit{attractiveness} (or equivalently is an attractive particle system) while our models are not known to have this property. This property states that if we have initial infected sets $A \subseteq B$ and $\xi_t^A$ and $\xi_t^B$ are the infected sets at time $t$ beginning from initial infected sets $A$ and $B$ respectively, then there exists a coupling between $\xi_t^A$ and $\xi_t^B$ such that $\xi_t^A \subseteq\xi_t^B$ for all $t \geq 0$. Attractiveness is an important ingredient in the proofs of many results about the classical contact process. Neither the contact process with isolation nor the contact process with vigilance is known to be an attractive particle system, and this presents technical challenges in deriving our results. 

We study the contact process with isolation and contact process with vigilance on configuration model random graphs where the degree distribution follows a power law distribution with power law exponent $\gamma > 3$. Our main result on the contact process with isolation shows that this model, like the classical contact process, exhibits at least stretched exponential survival for all choices of $\lambda$ and $\alpha$ on these graphs. This is somewhat surprising because \cite{CSW} and \cite{LamSIRS} previously showed that similar models exhibit polynomial survival on the star graph. By contrast, the classical contact process survives for exponentially long on star graphs, and this property is fundamental to proving that the classical contact process survives exponentially long for all $\lambda > 0$ on high degree random graphs~\cite{bhamidi2021survival,ChatterjeeDurrett2009,mountford2016exponential}. 

Our results for the isolation model also apply to another model for epidemic spread and forest fires that incorporates recurring infections and temporary immunity after infections. This model is known as the SIRS epidemic model. Some of its aspects have been studied in the probability literature \cite{Mollison1986, MollisonKuulasmaa, CoxDurrett} 
 using rigorous analytic methods and physics literature \cite{PSV2015, Silvaetal, Ferreiraetal} using non-rigorous approaches. Although the SIRS model is a more natural model for recurring infections compared to the SIS model,  only a few sporadic results have been obtained rigorously due to the difficulty posed by lack of monotonicity property described above. To the best of our knowledge, all previously published rigorous results \cite{DurrettNeuhauser, Grimmett1998} for the SIRS model are for homogeneous graphs and lattices. In this paper, we have been able to make a break through and obtain rigorous result involving the SIRS model on large finite heterogeneous graphs. Our result refutes the claim in the physics literature \cite{Ferreiraetal} about having a non-vanishing epidemic threshold for the SIRS model on heterogeneous random graphs

In contrast with the contact process with isolation and SIRS models, we show that the contact process with vigilance exhibits a phase transition in $\lambda$ for every fixed $\alpha$ between linear survival time and exponential survival time on graphs satisfying the isoperimetric inequality given in Assumption~\ref{assumption:isoperimetric}. These graphs include configuration model random graphs where the degree distribution follows a power law distribution with power law exponent $\gamma > 3$ and minimum degree $3$. We therefore observe that an individual versus a community-influenced response to an epidemic can have a substantial impact on survival and extinction of the epidemic.

\section{Main Results}
We define the \textbf{contact process with isolation} as follows. Let $G = (V,E)$ be a graph with vertices $V$ and edges $E$. Let $\xi_t(v)$ denote the state of vertex $v$ at time $t$, where $\xi_t(v) \in \{0,1,-1\}$ where state $0$ is healthy, state $1$ is infectious (or infected), and state $-1$ is isolating. Let $|\xi_t|$ denote the size of the infected set at time $t$. Given an initial condition $\xi_0 \in \{0,1,-1\}^V$ on the states of the vertices in $G$ and parameters $\lambda$ and $\alpha$, the process evolves according to the following rules:

\begin{enumerate}
    \item $\xi_t(v)$ goes from $0 \rightarrow 1$ at rate $\lambda\cdot | \{w\sim v : \xi_{t}(w) = 1\}|$.
    \item $\xi_t(v)$ goes from $1 \rightarrow -1$ at rate $\alpha$.
    \item $\xi_t(v)$ goes from $\{1,-1\} \rightarrow 0$ at rate $1$.
\end{enumerate}
Here rate means that the time to event follows an Exponential distribution with the given rate parameter. The contact process with isolation is intended to model an individually-driven response to an epidemic.

We define the \textbf{contact process with vigilance} as follows. Let $G = (V,E)$ be a graph with vertices $V$ and edges $E$. At any given time, each vertex $v$ is in one of three states: $1$ (infectious), $0$ (healthy), or $-1$ (isolated). Let $\zeta_t(v)$ be the state of vertex $v$ at time $t$. Given the initial states of all vertices $\zeta_0\in \{0,1,-1\}^V$, an infection rate parameter $\lambda$ and a vigilance rate parameter $\alpha$, the process evolves according to the following rules.

\begin{enumerate}
    \item For $v$ such that $\zeta_t(v) = 0$, $\zeta_t(v) \rightarrow 1 \textrm{ at rate } \lambda |\{w \sim v:\zeta_t(w) = 1\}|$.
    \item For $v$ such that $\zeta_t(v) = 1$, $\zeta_t(v) \rightarrow 0 \textrm{ at rate } 1$.
    \item For $v$ such that $\zeta_t(v) = 1$, $\zeta_t(v) \rightarrow -1 \textrm{ at rate } \alpha |\{w \sim v:\zeta_t(w) = 0\}|$.
    \item For $v$ such that $\zeta_t(v) = -1$, $\zeta_t(v) \rightarrow 0 \textrm{ at rate } 1$.
\end{enumerate}

Again rate means that the time to event follows an Exponential distribution with the given rate parameter. The first rule states that each healthy vertex becomes infected at rate $\lambda$ times the number of infected neighbors it has. The second rule states each infected vertex recovers and becomes healthy at rate $1$. The third rule states that each infectious vertex is identified and isolated at rate $\alpha$ times the number of healthy neighbors is has. The fourth rule states that each isolated vertex becomes healthy and returns from isolation at rate $1$. The contact process with vigilance is intended to model a community-driven response to an epidemic.

We now state our main results.

\begin{Assumption}\label{assumption:powerlaw}
    Suppose $G = (V,E)$ is a configuration model random graph with $n$ vertices where the degree sequence $D_1, \ldots, D_n$ is sampled i.i.d.~from a power law distribution with
    $$
    \prob(D_i \ge k) \asymp \frac{1}{k^{\gamma-1}}
    $$
    with exponent $\gamma > 3$. Moreover, suppose $D_i\ge 3$, so $G$ is connected with high probability.
\end{Assumption}

\begin{Assumption}\label{assumption:isoperimetric}
    Suppose that $G = (V,E)$ is a graph such that for some $0<\epsilon' < \epsilon < 1/2$ there exists a $\delta > 0$ such that for any set of vertices $B$ satisfying $\epsilon'|V| \leq |B| \leq \epsilon |V|$,
    $$|\partial B| \geq \delta |B|,$$
    where $\partial B = \{v \in V\setminus B : \exists w\in B \text{ such that } \{w,v\}\in E\}$ denotes the external vertex boundary of~$B$.
\end{Assumption}

\begin{theorem}\label{thm:powerlaw}
    Let $G = (V,E)$ be a graph satisfying Assumption \ref{assumption:powerlaw} with $|V| = n$, and let $\tau = \inf\{t: |\{\xi_t=1\}| = 0\}$ be the time to extinction of the contact process with isolation on $G$. Starting from initial condition $\xi_0(v) = 1$ for all $v \in V$, then for any $\alpha > 0$ and $\lambda > 0$ and $\eta > 0$ we have $\mathbb{P}(\tau < e^{n^{1-\eta}}) \rightarrow 0$ as $n \rightarrow \infty$.
\end{theorem}

\begin{theorem}\label{thm:vigilance-extinction}
    Let $G = (V,E)$ be any graph with $|V| = n$ and let $\tau = \inf\{t: |\{\zeta_t = 1\}| = 0\}$ be the time to extinction of the contact process with vigilance on $G$. Fix $\alpha > 0$ and any $\lambda < \alpha$. Starting from initial condition $\zeta_0(v) = 1$ for all $v \in V$, there exists $C > 0$ such that $\prob(\tau > C n) \rightarrow 0$ as $n \rightarrow \infty$.
\end{theorem}

\begin{theorem}\label{thm:vigilance-survival}
    Let $G = (V,E)$ be a graph satisfying Assumption \ref{assumption:isoperimetric} for some $\epsilon > 0$ and $\delta > 0$ with $|V| = n$ and let $\tau = \inf\{t: |\{\zeta_t = 1\}| = 0\}$ be the time to extinction of the contact process with vigilance on $G$. Fix $\alpha > 0$. Starting from initial condition $\zeta_0(v) = 1$ for all $v \in V$, there exists $\lambda_0=\lambda_0(\alpha, \epsilon, \delta)$ and $C>0$  such that $\mathbb{P}(\tau < e^{Cn}) \rightarrow 0$ as $n \rightarrow \infty$ for any $\lambda > \lambda_0$.
\end{theorem}

In Theorem \ref{thm:iso}, we show that Assumption \ref{assumption:isoperimetric} is satisfied with high probability when $G$ is a graph satisfying Assumption \ref{assumption:powerlaw} where $\gamma > 3$ and the minimum degree is $3$. This leads to the following corollary that gives a direct comparison to Theorem \ref{thm:powerlaw} for the contact process with vigilance.

\begin{corollary}
        Let $G = (V,E)$ be a graph satisfying Assumption \ref{assumption:powerlaw}, and let $\tau = \inf\{t: |\{\zeta_t = 1\}| = 0\}$ be the time to extinction of the contact process with vigilance on $G$. Fix $\alpha > 0$. Starting from initial condition $\zeta_0(v) = 1$ for all $v \in V$, there exists $\lambda_0=\lambda_0(\alpha,\gamma)$ and $C>0$  such that $\mathbb{P}(\tau < e^{Cn}) \rightarrow 0$ as $n \rightarrow \infty$ for any $\lambda > \lambda_0$.
\end{corollary}

The result in Theorem \ref{thm:powerlaw} also holds for the SIRS model. We first define the SIRS model as follows. Let $G = (V,E)$ be a graph with vertices $V$ and edges $E$. Let $\Upsilon_t(v)$ denote the state of vertex $v$ at time $t$ where $\Upsilon_t(v) \in \{0,1,-1\}$ where state $0$ is healthy, state $1$ is infections, and state $-1$ is removed, which is equivalent to isolating in the isolation model. Let $|\Upsilon_t|$ denote the size of the infected set at time $t$. Given an initial condition $\Upsilon_0 \in \{0,1,-1\}^V$ on the states of the vertices and parameters $\lambda$ and $\alpha$, the process evolves according to the following rules:
    \begin{enumerate}
        \item $\Upsilon_t(v)$ goes from $0 \rightarrow 1$ at rate $\lambda \cdot|\{w \sim v: \Upsilon_t(w) = 1\}|$.
        \item $\Upsilon_t(v)$ goes from $1 \rightarrow -1$ at rate $\alpha$.
        \item $\Upsilon_t(v)$ goes from $-1 \rightarrow 0$ at rate $1$.
    \end{enumerate}

Again, rate means that the time to event follows an Exponential distribution with the given rate parameter. We now formally state the result.

\begin{corollary}\label{cor:sirs}
    Let $G = (V,E)$ be a graph satisfying Assumption \ref{assumption:powerlaw} with $|V| = n$, and let $\tau = \inf\{t: |\{\Upsilon_t=1\}| = 0\}$ be the time to extinction of SIRS process on $G$. Starting from initial condition $\Upsilon_0(v) = 1$ for all $v \in V$, then for any $\alpha > 0$ and $\lambda > 0$ and $\eta > 0$ we have $\mathbb{P}(\tau < e^{n^{1-\eta}}) \rightarrow 0$ as $n \rightarrow \infty$.
\end{corollary}

To see why this is true, we observe that Lemma \ref{lem:Remenik comparison} also applies to this parameterization of the SIRS model using the same comparison process. That is, given initial condition $\Upsilon_0(V) = \xi_0^R(V)$ and shared parameters $\lambda$ and $\alpha$, the analogous coupling demonstrates that
    $$
    \xi_t^R(v) = 1 \implies \Upsilon_t(v) = 1
    $$
    for all $v \in V$ and all $t > 0$. Because our strategy for proving Theorem \ref{thm:powerlaw} is to first prove the result for the comparison process and, like the isolation model, the SIRS model stochastically dominates the comparison process, Corollary~\ref{cor:sirs} follows for the SIRS model.

We remark that the key to our proof of Theorem~\ref{thm:powerlaw} is that the comparison process, which dominates the contact process with isolation and SIRS processes from below, survives for stretched-exponentially long on a ``star of stars'' subgraph (see Definition~\ref{definition:starstar}). As a result, if the power law exponent $\gamma$ is larger then $2$ (rather than $3$) we can still show survival for time $e^{n^{\epsilon}}$ for some $\epsilon$ depending on $\lambda, \alpha, \gamma$ for both the isolation and SIRS models. This follows by identifying a single subgraph isomorphic to a star of stars of order $m=n^{\beta}$  then following the first part of the proof of Theorem~\ref{thm:powerlaw}. We require $\gamma>3$ to find $n^{1-\nu}$ such subgraphs (see Lemma~\ref{lem:power-law-star}), though this is likely an artifact of our proof, which can surely be improved to accommodate smaller values of $\gamma$.

\section{Background and related results}

One key question in the study of the contact process concerns the survival time of the infection. Foundational work on this question included studying the contact process on infinite lattices and trees, where it is known that the contact process exhibits a phase transition in $\lambda$ between almost sure extinction and positive probability of indefinite survival. For a detailed introduction to and summary of results about the contact process, we direct the reader to \cite{liggett2013stochastic} and \cite{DurrettDoG}.


\textbf{SIS on large finite graphs and extinction-time notions.}
On any finite graph the SIS/contact process dynamics eventually reaches the absorbing all-healthy state, so the infinite-graph survival notion does not directly carry over.
Instead, a central quantity is the \emph{extinction time}, often studied from the worst-case initial condition where all nodes are infected.
Early work analyzed extinction-time behavior on homogeneous finite graphs such as $d$-dimensional tori and grids \cite{durrett1988contact,mountford1993metastable}.
More recent results address complex network models including the Barab\'asi--Albert preferential attachment model \cite{berger2005spread}, the configuration model \cite{ChatterjeeDurrett2009,Mountford2013,mountford2016exponential,bhamidi2021survival}, Erd\H{o}s--R\'enyi random graphs \cite{cator2021explicit,nam2022critical}, random regular graphs \cite{mourrat2016phase}, small-world graphs \cite{durrett2007two}, networks with community structure \cite{sivakoff2017contact-communities}, and general random graph families \cite{valesin2024contact}.
In this line of work, a finite-graph system is called \emph{supercritical} (resp.\ \emph{subcritical}) when the extinction time grows ``large'' (resp.\ ``small'') with the number of nodes $n$; typically ``large'' means exponential (or stretched-exponential) in $n$, whereas ``small'' is closer to logarithmic in $n$.

A consistent theme is that degree heterogeneity strongly affects persistence.
For instance, under extinction-time notions of criticality, it has been shown that $\lambda_c=0$ for the Barab\'asi--Albert model \cite{berger2005spread} and for configuration models with heavy-tailed degree distributions \cite{ChatterjeeDurrett2009}.
Moreover, for these heavy-tailed settings the extinction time can be stretched-exponential (and in some cases exponential) in $n$ even for arbitrarily small $\lambda>0$, with a metastable infected fraction maintained during the long persistence window \cite{Mountford2013,mountford2016exponential}.
In contrast, for degree distributions having exponential moments, recent results show that the corresponding $\lambda_c$ (in the extinction-time sense) is positive, implying that sufficiently weak infection dies out rapidly regardless of network size \cite{bhamidi2021survival}; in particular this includes random regular graphs and Erd\H{o}s--R\'enyi graphs \cite{mourrat2016phase,cator2021explicit,nam2022critical}.

\textbf{Thresholds from small initial prevalence and spectral mean-field baselines.}
A different and widely used notion of onset asks what happens when the process starts from a small number (or small fraction) of infected nodes.
In this setting one again observes two qualitative behaviors: (a) rapid decay to zero (extinction) or (b) growth to a positive metastable prevalence (persistence).
A foundational line of work links the onset of endemic SIS behavior to the spectral radius of the adjacency matrix.
Early eigenvalue-based arguments and mean-field analyses suggest a threshold of the form $\tau_c \approx 1/\lambda_1(A)$, where $\tau=\beta/\mu$ and $\lambda_1(A)$ denotes the largest eigenvalue (spectral radius) of $A$ \cite{PastorSatorras2001,Wang2003}.
The N-intertwined mean-field approximation (NIMFA) formalizes node-level dynamics that recover the same adjacency-spectral condition in a tractable ODE system \cite{VanMieghem2009}, and reviews summarize when such approximations are accurate and when they deviate \cite{PSV2015}.
These static-network baselines are the conditions our multiplex analysis must reduce to when temporal contacts are absent.

For highly heterogeneous networks, however, the effective onset behavior can differ from the simple $\tau_c\approx 1/\lambda_1(A)$ picture, and the probability literature provides complementary perspectives.
Using self-duality of the contact process \cite{liggett2013stochastic}, survival from a small seed can be related to metastable prevalence under all-infected initialization.
In heavy-tailed random-graph settings, this motivates scaling relations for the probability of long persistence from a small number of initially infected nodes and yields thresholds that may vanish with $n$ (up to logarithmic factors), consistent with the extinction-time viewpoint \cite{ChatterjeeDurrett2009,Mountford2013,mountford2016exponential,bhamidi2021survival}.
When one wishes to make this connection explicit, one may express a survival/metastable-density scaling $\phi(\lambda)$ (for power-law exponent $\alpha$) and infer how many seeds are needed to trigger persistence; the corresponding network-size scaling is then tied to eigenvalue/maximum-degree scalings in heavy-tailed graphs \cite{flaxman2005high,Chung2003}.
(We emphasize that these heavy-tail asymptotics depend on the random-graph model class and exponent regime; we cite them here primarily to highlight that multiple notions of ``threshold'' coexist in the literature and can behave differently in heterogeneous networks.)


\textbf{Models with additional vertex and edge behaviors.}
Recently, there has been increased interest in variant contact process models that include additional vertex and edge behaviors intended to model more complex disease dynamics, behaviors, and population structures. An early example is the adaptive SIS model proposed by \cite{GrossAdSIS} in which $SI$ edges randomly rewire. While this model has received considerable attention \cite{GrossReview, GuoEtAl, PipatsartEtAl, DemirelEtAl}, it has proven difficult to study rigorously, with the aforementioned work focusing on mean field approximations, moment closures, and computational results.

Some variants of the adaptive SIS model have proven more tractable. \cite{CSW} consider a model in which directed edges temporarily deactivate rather than rewrite and show that this model exhibits a phase transition like that of the classical contact process on $\mathbb{Z}, \mathbb{Z} \mod n$ and exhibits asymptotic survival time polynomial in the size of the graph on the star graph. \cite{TuncEtAl} study a similar model where bidirectional edges temporarily deactivate and derive an epidemic threshold on the complete graph using differential equation approximations. \cite{JacobMorters} study a model in which edges rewrite independently of vertex states and show that this model has a phase transition in $\lambda$ for certain power law degree distribution graphs and sufficiently fast rewiring.

The evoSIR model, which considers the random rewiring of the adaptive SIS model in the context of an underlying SIR model, has also been studied. In general, SIR models tends to be easier to study than SIS models because in SIR models vertices are removed from the graph upon recovery and so each vertex can become infected at most once. Indeed, more is known about the evoSIR model, including the existence of phase transitions in $\lambda$ and their continuity behavior on Erdos-Renyi graphs \cite{JiangEtAl} and configuration model graphs \cite{DurrettYao}.

Other variants of the contact process, including the isolation and vigilance models we study in this work, add additional vertex states rather than edge behaviors. \cite{Remenik} studies an ecologically-motivated model that we can translate into the language of our isolation model. In this model all vertices, rather than only infectious vertices, isolate at rate $\alpha$ and return from isolation at rate $\alpha \delta$. He shows this model is monotonic in $\lambda$, $\alpha$, and $\delta$ individually, although changing multiple parameters can lead to incomparable processes. He also shows this model has phase transitions on $\mathbf{Z}$ in both $\lambda$ and $\delta$. For sufficiently small $\delta$, the infection dies out almost surely for all $\lambda$, while for larger $\delta$ the infection has positive probability of persisting indefinitely only when $\lambda$ is sufficiently large. We make a more detailed comparison of our isolation model with this model in Section \ref{subsec:isolation-overview}.

Another related model is the SIRS model, in which after recovering, individuals become temporarily immune to reinfection. Like the isolation model, the SIRS model is not monotonic in $\lambda$ for reasons analogous to those demonstrated in Figure \ref{fig:isolation}.  This has made rigorous study of the SIRS model difficult. However, some results have been derived. \cite{DurrettNeuhauser} show that for sufficiently large $\lambda$, the SIRS model has a nontrivial stationary distribution on $\mathbb{Z}^2$. However, their results rely on isoperimetric properties specific to $\mathbb{Z}^2$ and so do not generalize to other graphs. 

More recently, \cite{LamSIRS} showed that the survival time of the SIRS model on the star graph is asymptotically polynomial in the size of the graph. Because the result of \cite{mountford2016exponential} that the contact process has survival time $e^{\Theta(n)}$ on power law degree distribution random graphs for all $\lambda > 0$ relies heavily on the fact the contact process has survival time $e^{\Theta(n)}$ on star graphs for all $\lambda > 0$, the results of \cite{LamSIRS} and \cite{CSW} pose an interesting question of whether these variant models exhibit qualitatively different behavior from that of the classical contact process on power law degree distribution random graphs. Conversely, \cite{FriedrichSIRS} show that the expected survival time of the SIRS process is exponential in the size of the graph for certain expander graphs, suggesting one might expect similar behavior on random graphs with power law degree distributions. Theorem \ref{thm:powerlaw} answers this question for our isolation model, and by Corollary~\ref{cor:sirs}, this also holds for the SIRS model. 
Very recently, several months after our preprint \cite{OldIsoVig} was posted on arXiv, three independent preprints \cite{lam2026supercriticalitysirsrandomnetworks, 
he2026waningimmunityfailsrestore, gobel2026}  concerning related aspects of the SIRS model on networks appeared. Some of the results in these works overlap with a corollary of the main results already obtained in \cite{OldIsoVig}. 

\section{Graphical construction}\label{sec:graph}

A useful tool for studying the contact process and its variants is a graphical construction sometimes called the Harris construction. For the classical contact process, this construction is defined as follows. For a graph $G$, consider the spacetime region $G \times [0,\infty)$. At each vertex $v$ and directed edge $e$ of the graph, we will have a Poisson process of temporal marks with appropriate rates, defined as follows.

\begin{enumerate}
    \item[(C1)] For each directed edge $e\in G$, define a Poisson process on $\{e\}\times [0,\infty)$ with intensity $\lambda$ that generates infection arrows.
    \item[(C2)] For each vertex $v\in G$, define a Poisson process on $\{v\}\times [0,\infty)$ with intensity $1$ that generates recovery dots.
\end{enumerate}

We can then realize the process on $G \times [0,\infty)$ given some initial state $\xi_0\in \{0,1\}^V$, where $1$ indicates an infected vertex and 0 a healthy vertex, using these marks by doing the following.

\begin{enumerate}
    \item Label infected vertices $1$ (blue in Figure~\ref{fig:classical}) vertically in time until a recovery dot is reached.
    \item When an infection arrow is observed, if the source vertex is infected (1, blue), infect the target vertex and repeat step 1. for this newly infected vertex.
\end{enumerate}

The set of vertices in $G$ infected at time $t$ is exactly those vertices reached by a blue path at time $t$. See Figure \ref{fig:classical} for an example. This graphical construction can be used to demonstrate that the classical contact process is attractive. Adding additional vertices to the initial infected set can only increase the number of vertices reached by a blue path at time $t$.

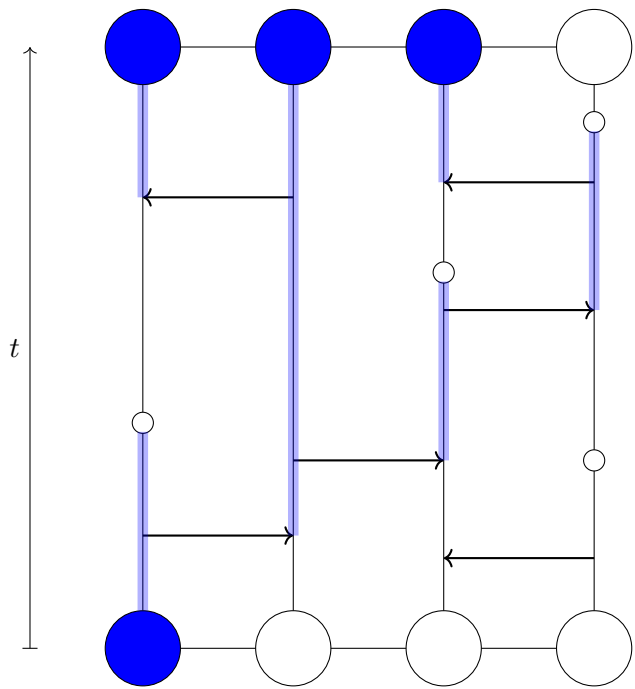
\begin{figure}\label{fig:classical}
    \centering
\begin{tikzpicture}
  \def\n{4}                 
  \def\radius{0.5}          
  \def\spacing{2}           
  \def\lineheight{8}      


    \draw[fill=blue, draw=black] (0*\spacing, 0) circle (\radius);
    \draw (1*\spacing, 0) circle (\radius);
    \draw (2*\spacing, 0) circle (\radius);
    \draw (3*\spacing, 0) circle (\radius);

    \draw[fill=blue, draw=black] (0*\spacing, \lineheight) circle (\radius);
    \draw[fill=blue, draw=black] (1*\spacing, \lineheight) circle (\radius);
    \draw[fill=blue, draw=black] (2*\spacing, \lineheight) circle (\radius);
    \draw (3*\spacing, \lineheight) circle (\radius);

\foreach \i in {0,...,3} {
    \pgfmathsetmacro{\x}{\i * \spacing}
    \draw (\x, \radius) -- (\x, \lineheight - \radius);
  }

  \foreach \i in {0,...,2} {
    \pgfmathsetmacro{\xL}{\i * \spacing + \radius}
    \pgfmathsetmacro{\xR}{(\i + 1) * \spacing - \radius}

    \draw (\xL, 0) -- (\xR, 0);

    \draw (\xL, \lineheight) -- (\xR, \lineheight);
  }

    \draw[->] (-1.5, 0) -- (-1.5, 8) node[pos = .5, left] {$t$};
    \draw (-1.6,0) -- (-1.4,0);

  \filldraw[fill=white, draw=black] (4, 5.0) circle (4pt);
  \filldraw[fill=white, draw=black] (6, 2.5) circle (4pt);
  \filldraw[fill=white, draw=black] (6, 7.0) circle (4pt);
  \filldraw[fill=white, draw=black] (0, 3) circle (4pt);

  \draw[->, thick] (0, 1.5) -- (2, 1.5);
  \draw[->, thick] (2, 2.5) -- (4, 2.5);
  \draw[->, thick] (4, 4.5) -- (6, 4.5);
  \draw[->, thick] (6, 6.2) -- (4, 6.2);
  \draw[->, thick] (2, 6.0) -- (0, 6.0); 
  \draw[->, thick] (6, 1.2) -- (4, 1.2);

  \draw[line width=4pt, blue, opacity=0.3] (0,.5) -- (0,2.88);
   \draw[line width=4pt, blue, opacity=0.3] (2,1.5) -- (2,7.5);
    \draw[line width=4pt, blue, opacity=0.3] (4,2.5) -- (4,4.88);
    \draw[line width=4pt, blue, opacity=0.3] (0,6.0) -- (0,7.5);
    \draw[line width=4pt, blue, opacity=0.3] (6,4.5) -- (6,6.88);
    \draw[line width=4pt, blue, opacity=0.3] (4,6.2) -- (4,7.5);

\end{tikzpicture}
    \caption{An example of the graphical construction for the classical contact process. Blue vertices are infected, and white vertices are susceptible. We follow the blue paths to determine which vertices are infected at time $t$.}
\end{figure}

We can define a similar graphical construction for the contact process with isolation by adding Poisson processes to generate marks that govern isolation transitions. 

\begin{enumerate}
    \item[(I3)] For each vertex $v \in G$, define a Poisson process on $\{v\} \times [0,\infty)$ with intensity $\alpha$ that generates isolation crosses.
\end{enumerate}

We can then realize the process on $G \times [0,\infty)$ given some initial state $\xi_0\in \{-1,0,1\}^V$, where $1$ indicates an infected vertex, 0 a healthy vertex, and $-1$ and isolated vertex using these marks by doing the following.

\begin{enumerate}
    \item Label infected vertices $1$ (blue in Figure~\ref{fig:classical}) vertically in time until a recovery dot or an isolation cross is reached. If a recovery dot is reached first, return the vertex to the health (0) state and remove its color.
    \item When an isolation cross is observed on a ($1$, blue in Figure~\ref{fig:classical}) vertex, change its state to isolated (-1) and color the vertex red vertically in time until a recovery dot is reached. When a recovery dot is reached first, return the vertex to the health (0) state and remove its color.
    \item When an infection arrow is observed, if the source vertex is infected (1, blue), infect the target vertex and repeat step 1. and 2. for this newly infected vertex.
\end{enumerate}

Again, the set of vertices in $A$ infected at time $t$ is exactly those vertices reached by a blue path at time $t$. We can also see from this construction why the argument used to show the classical contact process is attractive fails for the contact process with isolation. Adding additional vertices to the initial infected set may cause additional isolations and larger red regions. As in the example shown in Figure \ref{fig:isolation}, this can lead to a smaller infected set at time $t$.

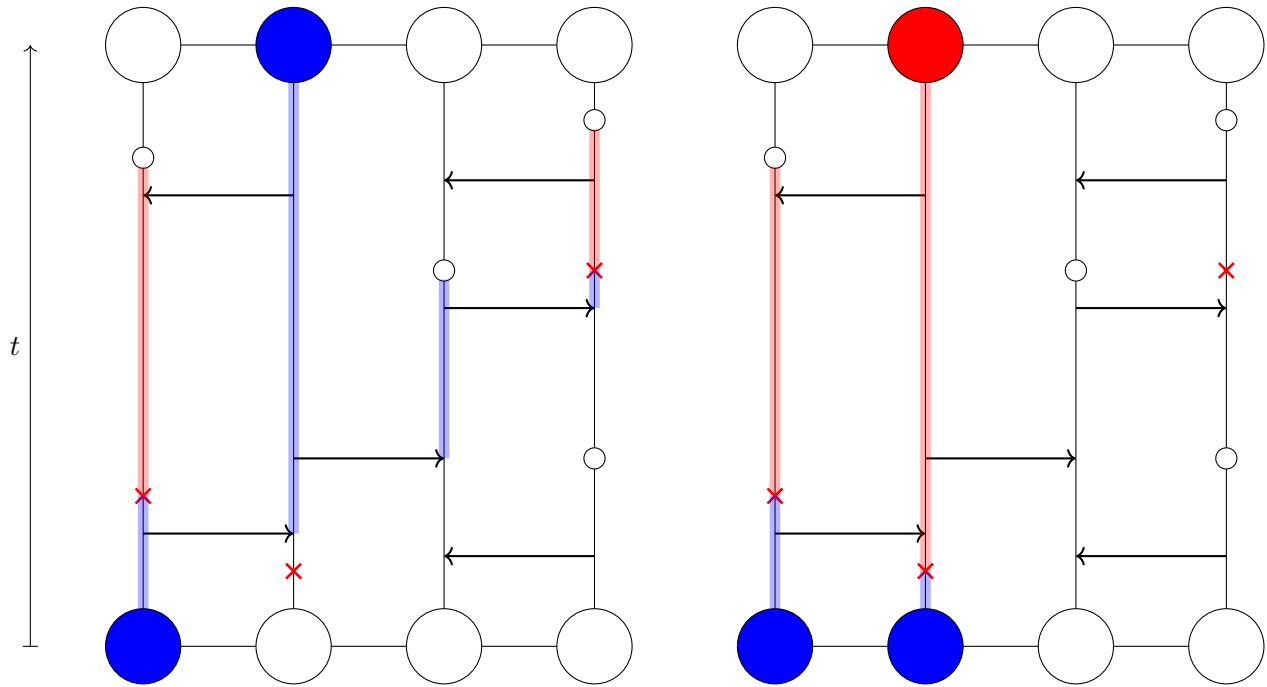
\begin{figure}\label{fig:isolation}
    \centering
\begin{tikzpicture}[scale=0.75]
  \def\n{4}                 
  \def\radius{0.5}          
  \def\exsize{0.75} 
  \def\spacing{2}           
  \def\lineheight{8}  


    \draw[fill=blue, draw=black] (0*\spacing, 0) circle (\radius);
    \draw (1*\spacing, 0) circle (\radius);
    \draw (2*\spacing, 0) circle (\radius);
    \draw (3*\spacing, 0) circle (\radius);

    \draw[fill=blue, draw=black] (4.2*\spacing, 0) circle (\radius);
    \draw[fill=blue, draw=black] (5.2*\spacing, 0) circle (\radius);
    \draw (6.2*\spacing, 0) circle (\radius);
    \draw (7.2*\spacing, 0) circle (\radius);

    \draw (0*\spacing, \lineheight) circle (\radius);
    \draw[fill=blue, draw=black] (1*\spacing, \lineheight) circle (\radius);
    \draw (2*\spacing, \lineheight) circle (\radius);
    \draw (3*\spacing, \lineheight) circle (\radius);

        \draw (4.2*\spacing, \lineheight) circle (\radius);
    \draw[draw=black] (5.2*\spacing, \lineheight) circle (\radius);
    \draw[line width=4pt, red, opacity=0.3] (5.2*\spacing, \lineheight) circle (\radius);
    \draw (6.2*\spacing, \lineheight) circle (\radius);
    \draw (7.2*\spacing, \lineheight) circle (\radius);

\foreach \i in {0,...,3} {
    \pgfmathsetmacro{\x}{\i * \spacing}
    \draw (\x, \radius) -- (\x, \lineheight - \radius);
  }

\foreach \i in {4,...,7} {
    \pgfmathsetmacro{\x}{(\i+.2) * \spacing}
    \draw (\x, \radius) -- (\x, \lineheight - \radius);
  }

  \foreach \i in {0,...,2} {
    \pgfmathsetmacro{\xL}{\i * \spacing + \radius}
    \pgfmathsetmacro{\xR}{(\i + 1) * \spacing - \radius}

    \draw (\xL, 0) -- (\xR, 0);

    \draw (\xL, \lineheight) -- (\xR, \lineheight);
  }

    \foreach \i in {4,...,6} {
    \pgfmathsetmacro{\xL}{(\i+.2) * \spacing + \radius}
    \pgfmathsetmacro{\xR}{(\i + 1.2) * \spacing - \radius}

    \draw (\xL, 0) -- (\xR, 0);

    \draw (\xL, \lineheight) -- (\xR, \lineheight);
  }

    \draw[->] (-1.5, 0) -- (-1.5, 8) node[pos = .5, left] {$t$};
    \draw (-1.6,0) -- (-1.4,0);

    \def\crosssize{0.2}
    \draw[line width=1pt, red] 
  (2 - \crosssize, 1 - \crosssize) -- (2 + \crosssize, 1 + \crosssize)
  (2 - \crosssize, 1 + \crosssize) -- (2 + \crosssize, 1 - \crosssize);

      \draw[line width=1pt, red] 
  (0 - \crosssize, 2 - \crosssize) -- (0 + \crosssize, 2 + \crosssize)
  (0 - \crosssize, 2 + \crosssize) -- (0 + \crosssize, 2 - \crosssize);

      \draw[line width=1pt, red] 
  (6 - \crosssize, 5 - \crosssize) -- (6 + \crosssize, 5 + \crosssize)
  (6 - \crosssize, 5 + \crosssize) -- (6 + \crosssize, 5 - \crosssize);

      \draw[line width=1pt, red] 
  (2+8.4 - \crosssize, 1 - \crosssize) -- (2+8.4 + \crosssize, 1 + \crosssize)
  (2+8.4 - \crosssize, 1 + \crosssize) -- (2+8.4 + \crosssize, 1 - \crosssize);

      \draw[line width=1pt, red] 
  (0+8.4 - \crosssize, 2 - \crosssize) -- (0+8.4 + \crosssize, 2 + \crosssize)
  (0+8.4 - \crosssize, 2 + \crosssize) -- (0+8.4 + \crosssize, 2 - \crosssize);

      \draw[line width=1pt, red] 
  (6+8.4 - \crosssize, 5 - \crosssize) -- (6+8.4 + \crosssize, 5 + \crosssize)
  (6+8.4 - \crosssize, 5 + \crosssize) -- (6+8.4 + \crosssize, 5 - \crosssize);

  \filldraw[fill=black, draw=black] (4, 5.0) circle (5pt);
  \filldraw[fill=black, draw=black] (6, 2.5) circle (5pt);
  \filldraw[fill=black, draw=black] (6, 7.0) circle (5pt);
  \filldraw[fill=black, draw=black] (0, 6.5) circle (5pt);

    \filldraw[fill=black, draw=black] (4+8.4, 5.0) circle (5pt);
  \filldraw[fill=black, draw=black] (6+8.4, 2.5) circle (5pt);
  \filldraw[fill=black, draw=black] (6+8.4, 7.0) circle (5pt);
  \filldraw[fill=black, draw=black] (0+8.4, 6.5) circle (5pt);

  \draw[->, thick] (0, 1.5) -- (2, 1.5);
  \draw[->, thick] (2, 2.5) -- (4, 2.5);
  \draw[->, thick] (4, 4.5) -- (6, 4.5);
  \draw[->, thick] (6, 6.2) -- (4, 6.2);
  \draw[->, thick] (2, 6.0) -- (0, 6.0); 
  \draw[->, thick] (6, 1.2) -- (4, 1.2);

    \draw[->, thick] (0+8.4, 1.5) -- (2+8.4, 1.5);
  \draw[->, thick] (2+8.4, 2.5) -- (4+8.4, 2.5);
  \draw[->, thick] (4+8.4, 4.5) -- (6+8.4, 4.5);
  \draw[->, thick] (6+8.4, 6.2) -- (4+8.4, 6.2);
  \draw[->, thick] (2+8.4, 6.0) -- (0+8.4, 6.0); 
  \draw[->, thick] (6+8.4, 1.2) -- (4+8.4, 1.2);

  \draw[line width=4pt, blue, opacity=0.3] (0,.5) -- (0,2);
  \draw[line width=4pt, red, opacity=0.3,decorate,decoration=zigzag] (0,2) -- (0,6.38);
   \draw[line width=4pt, blue, opacity=0.3] (2,1.5) -- (2,7.5);
    \draw[line width=4pt, blue, opacity=0.3] (4,2.5) -- (4,4.88);
    \draw[line width=4pt, blue, opacity=0.3] (6,4.5) -- (6,5);
    \draw[line width=4pt, red, opacity=0.3, decorate,decoration=zigzag] (6,5) -- (6,6.88);

    \draw[line width=4pt, blue, opacity=0.3] (2+8.4,.5) -- (2+8.4,1);
      \draw[line width=4pt, blue, opacity=0.3] (0+8.4,.5) -- (0+8.4,2);
  \draw[line width=4pt, red, opacity=0.3, decorate,decoration=zigzag] (0+8.4,2) -- (0+8.4,6.38);
   \draw[line width=4pt, red, opacity=0.3, decorate,decoration=zigzag] (2+8.4,1.0) -- (2+8.4,7.5);

\end{tikzpicture}
    \caption{In the isolation model adding additional vertices to the initial infected set can lead to a smaller final infected set. Blue vertices are infected, white vertices are susceptible, and red (squiggly, highlighted) vertices are isolated.} 
\end{figure}

\section{Isolation model}\label{sec:isolation}
\subsection{Overview}\label{subsec:isolation-overview}
Our goal in the first part of this paper is to prove Theorem \ref{thm:powerlaw}, and our strategy is to employ an argument used by Chatterjee and Durrett in~\cite{ChatterjeeDurrett2009}. They show that the classical contact process ($\alpha=0$, $\lambda>0$) survives for time $e^{n^{1-\eta}}$ with high probability by showing that high-degree nodes can sustain the infection for exponentially long in their degree, which is long enough to push the infection to other high-degree nodes. They show that the number of these high-degree nodes whose neighborhoods are highly infected dominate a random walk with positive drift, which leads to stretched-exponential survival on power-law degree distributed random graphs. 

There are two main challenges in employing this proof strategy. First, their argument relies crucially on the attractiveness of the classical contact process. The contact process with isolation does not have this property, so our first order of business is to prove that the infected set in the contact process with isolation dominates a comparison process in which all vertices, not just infected vertices, enter isolation at rate $\alpha$. This process is equivalent to a special case of the model studied by Remenik~\cite{Remenik}, where his parameter $\delta = 1/\alpha$.  Unlike the isolation model, the comparison model is attractive in the same sense as the classical contact process, and Remenik~\cite{Remenik} notes this is true of his model generally.

The second main challenge is that the contact process with isolation only survives for a polynomial amount of time on a star graph, which is not necessarily long enough for a high degree node to push the infection to other high degree nodes. This can be proved via the same methods used to study the contact process with avoidance~\cite{CSW} and the SIRS model~\cite{LamSIRS} on stars. To get around this obstacle, we introduce a structure that we call a star of stars of order $m$, which is essentially a regular tree of depth $2$ and degree $m$. The majority of our proof is then devoted to showing that the comparison process survives for time $e^{m^\beta}$ with probability at least $1-e^{-m^\rho}$ on a star of stars of order $m$. This is a consequence of Lemmas~\ref{lemma:one-phase-lit} and~\ref{lem:onereturn}. Once this is established, we show that power-law degree distributed random graphs contain many subgraphs isomoriphic to a star of stars of order $m=n^{\epsilon}$, and our proof is completed along the same lines as in~\cite{ChatterjeeDurrett2009}.

\subsection{Comparison process}
We will now formally define the comparison process and establish stochastic domination before briefly discussing why the comparison process is attractive. We define the comparison process as follows. Let $G = (V,E)$ be a graph and denote the state of a vertex $v$ at time $t$ by $\xi^R_t(v)\in \{-1,0,1\}$ where state $1$ is infectious, state $0$ is healthy (and susceptible), and state $-1$ is isolating. Given an initial state $\xi_0^R$, this comparison model evolves according to the following rules

\begin{enumerate}
    \item If $\xi^R_t(v) = 0$, then $\xi^R_t(v) \mapsto 1$ at rate $\lambda\cdot | \{w\sim v : \xi_{t}(w) = 1\}|$.
    \item If $\xi^R_t(v) \in \{0,1\}$, then $\xi^R_t(v) \mapsto -1$ at rate $\alpha$.
    \item If $\xi^R_t(v) \in \{1,-1\}$, then $\xi^R_t(v) \mapsto 0$ at rate $1$.
\end{enumerate}

We now show that the infected set in the comparison process is dominated by the infected set in the contact process with isolation. We note that the analogous statement also holds for the SIRS process.
\begin{lemma}\label{lem:Remenik comparison}
    The infected set in the contact process with isolation stochastically dominates the infected set in the comparison process. That is, given initial conditions $\xi_0 = \xi^R_0$ and shared parameters $\lambda$ and $\alpha$, there exists a coupling such that

    $$\xi_t^R(v) = 1 \quad \text{ implies }\quad  \xi_t(v) = 1$$

    for all $v \in V$ and all $t > 0$.
\end{lemma}

\begin{proof}
    We can demonstrate an appropriate coupling using the graphical construction defined via the Poisson processes of marks (C1), (C2), and (I3) in Section \ref{sec:graph}. 


We can couple the isolation model with the comparison process by generating the same realization of symbols in the graphical constructions for both and considering an arbitrary initial configuration. However, the rules for realizing the process differ for the isolation model and the comparison process. In the isolation model:

\begin{enumerate}
    \item Color infected (state 1) vertices blue vertically in time until either a recovery dot or isolation cross is reached. If a recovery dot is reached first, return the vertex to the healthy (0) state and remove its color.
    \item When an isolation cross is observed on an infected (1, blue) vertex, change its state to isolated (-1) and color the vertex red vertically in time until a recovery dot is reached.  When a recovery dot is reached, return the vertex to the healthy (0) state and remove its color.
    \item When an infection arrow is observed, if the source vertex is infected (1, blue) and the target vertex is not isolated, infect the target vertex and repeat steps 1.-2. for this newly infected vertex.
\end{enumerate}

We can then follow the blue paths to determine which vertices in $G$ are infected at time $t$. Note that the path of the infection cannot travel through the red regions indicating isolated vertices.

To construct the comparison process on the same realization of marks in the graphical construction, step 2. above is replaced by $2'$ below.
\begin{enumerate}
    \item[$2'$.] When an isolation cross is observed on any vertex, color the vertex red vertically in time until a recovery dot is reached, at which time its state reverts to healthy (0) and its color removed.
\end{enumerate}
Thus, any red region that exists in the isolation model will exist in the comparison model, as the isolation cross that initiated the red region in the isolation model will also initiate a red region in the comparison model, and the red regions in both models will continue until they encounter the next recovery dot, which occurs at the same time for both processes. The comparison model may have additional red regions not present in the isolation model.

We now consider that if we were to ignore the red regions, the path of infection in both models will be the same, as they share the same infection arrows and recovery dots. However, we truncate any infection paths that attempt to cross a red region. Thus, if the red regions in the isolation model are contained in the red regions in the comparison model, any infection path in the comparison model that is not truncated by a red region is also not truncated by a red region in the isolation model. This demonstrates a coupling in which under any shared realization of marks in the graphical construction and any initial configuration on $G \times [0,\infty)$, for every vertex $v$ and time $t$,

$$\xi_t^R(v)=1 \implies \xi_t(v)=1.$$

That is, the set of infected vertices in the comparison model at time $t$ is contained in the set of infected vertices in the isolation model at time $t$.
\end{proof}

We now observe that the comparison model is attractive. The isolation model is not attractive because additional infections can create additional isolations that might block infection paths that would otherwise successfully persist from time $0$ to time $t$ in our graphical construction. However, in the comparison model, the isolated vertices at any time are determined completely by the isolation crosses and recovery dots. The configuration of infection arrows has no effect on when isolations start and end. Thus, adding more infected vertices to the initially infected set or adding more infection arrows to a realization of the graphical construction can only ever increase the number of infection paths that persist from time $0$ to time~$t$. A key consequence that we will use repeatedly is that if the infection persists using only the marks of the graphical construction internal to a subgraph of a graph $G$, then it persists in that subgraph using the full graphical construction.

\subsection{Comparison process on a star of stars} We begin by studying the comparison process on what we call a \textbf{star of stars graph of order $m$}. We define this structure as follows.

\begin{definition}\label{definition:starstar}
A graph $G = (V,E)$ is called a \textbf{star of stars graph of order $m$} if it is a rooted tree with the following properties.
    \begin{itemize}
    \item The root has $m$ children. Call the root the \textbf{center} vertex.
    \item Call the $m$ children of the center \textbf{hubs}. Each hub has $m$ children, which we call \textbf{leaves}.
\end{itemize}
\end{definition}

We will call the subgraph induced by the center and hubs the \textbf{center star} and the subgraph induced by each hub and its leaves a \textbf{hub star}. We define the \textbf{process internal to a hub star} to mean the realization of the process using only marks in the graphical construction internal to that hub star. 

Our goal will be to show that for any $\alpha, \lambda > 0$ there are constants $\beta, \rho>0$ such that the infection survives on a star of stars of order $m$ for time at least $e^{m^{\beta}}$ with probability at least $1-e^{-m^\rho}$. This is a consequence of Lemmas~\ref{lemma:one-phase-lit} and~\ref{lem:onereturn}.


We begin by giving a high-level overview of our proof strategy. If we consider the state of the center, we can classify the process as being in one of three phases at any given time:

\begin{itemize}
    \item The one-phase, when the center is infected.
    \item The zero-phase, when the center is healthy.
    \item The minus one-phase, when the center is isolating.
\end{itemize}

During the one-phase, the center is spreading the infection to the hubs, potentially reintroducing it to any hub stars where it might have previously gone extinct. During the zero-phase, any hub star can potentially reinfect the center, assuming its own infection has not gone extinct locally. During the minus one-phase, the hub stars behave independently as stars of degree $m$, and can potentially see their infections go extinct locally.

From this, we can broadly see three reasons the infection might die out on a star of stars.

\begin{enumerate}
    \item During a one-phase, the center does not infect enough hubs.
    \item During a zero-phase, the non-extinct hub stars do not reinfect the center quickly enough.
    \item During a minus one-phase, too many hub stars go extinct locally.
\end{enumerate}

Our goal will be to show that each of these is unlikely. However, 1. presents a problem. Namely, the one-phase could be very short, and in this case we would not expect the center to infect many hubs. Thus, we will define long one-phases to be those that last for at least time $1/(1+\alpha)$ and change item 1.~above to

\begin{enumerate}
    \item During a long one-phase, the center does not infect enough hubs.
\end{enumerate}

Since we are interested in long one-phases, we combine 2. and 3. into

\begin{enumerate}
    \setcounter{enumi}{1}
    \item Between long one-phases, too many hub stars go extinct locally.
\end{enumerate}

Having set forth our strategy, we now proceed to examine the process during the different center phases. 

\subsection{The one-phase}\label{subsec:isolation-one-phase}

Recall that a one-phase occurs when the center becomes infected and ends when the center stops being infected. Let $S_i$ be the time of the $i$th reinfection of the center, and let $T_i$ be the next time after $S_i$ that the center is not infected. Note that $S_i$ and $T_i$ are stopping times, and on the event $S_i < \infty$, $T_i - S_i$ is exponentially distributed and independent of the $\sigma$-field generated by all marks in the graphical construction outside the center during $[S_i,T_i]$. We will call the $i$th one-phase \textbf{long} if $T_i - S_i > 1/(1+\alpha)$. 

For the process to survive, we need the center to spread infection to the hub stars sufficiently so that enough hub stars can survive independently of the center until the next long one-phase. To this end, we will call a hub star \textbf{non-extinct} whenever either its hub is infected and or has at least $1$ infected leaf.

Our goal in studying the one-phase is to show that there is some $\eta > 0$ such that at the end of a long one-phase, there are at least $\eta m$ non-extinct hub stars with high probability irrespective of the configuration at the start of the long one-phase. Since $S_i$ is a stopping time, the strong Markov property allows us to consider the process starting from time $0$ with the center infected and show that, conditioned on this initial one-phase being long, the center is likely to spread the infection to enough hub stars. We formalize this in the following lemma.

\begin{lemma}\label{lemma:one-phase-lit}
    Fix $\alpha, \lambda > 0$ and suppose we start the process at time $0$ from an initial configuration $\xi^R_0$ such that the center is infected. Let $T_0$ be the first time that the center is not infected, let $G_0$ be the event that $T_0 \ge 1/(1+\alpha)$, and let $N_t$ be the number of non-extinct hub stars at time $t=1/(1+\alpha)$. Then there exists $\eta > 0$ depending on $\alpha$ and $\lambda$ but not on $m$ such that

    $$\probg(N_t < \eta m) \leq \exp\bigg(\frac{-\eta m}{4} \bigg)$$
    where $\probg$ means the probability conditional on the event $G_0$.
\end{lemma}

\begin{proof}
Consider a fixed hub. By monotonicity, on the event $G_0$ we set $T_0 = t$ and observe that the following sequence of events ensures the hub is infected at time $T_0$.

\begin{enumerate}
    \item [(a)] Hub vertex has a recovery dot in $(0,t/3)$
    \item [(b)] Hub vertex has no recovery dots in $(t/3, t)$
    \item [(c)] Hub vertex has no isolation crosses in $(0,t)$
    \item [(d)] Center has an infection arrow to the hub in $(t/3,2t/3)$
\end{enumerate}

Call the intersection of events (a)--(d) above $G_1$. We can use the independence of these events to compute the probability of $G_1$

\begin{equation*}
    \probg(G_1) = (1 - e^{-t/3})(e^{-2t/3})(e^{-\alpha t})(1-e^{-\lambda t/3})=: \gamma_1.
\end{equation*}

Since $t = \frac{1}{1+\alpha}$, we have $\gamma_1>0$ depends on $\alpha$ and $\lambda$ but not on $m$. This is enough to ensure the hub star is non-extinct at time $T_0$. Observe that since we have set $T_0 = t$, $N_t$ is the number of non-extinct hub stars at time $T_0$. Then on the event $G^0$, $N_t$ stochastically dominates $Y$ where

$$Y \sim \textrm{Binomial}(m, \gamma_1).$$

Choose $\eta > 0$ depending on $\lambda$ and $\alpha$ such that $E(Y) \geq 2 \eta m$. By a Chernoff bound,
$$
\probg(N_t <  \eta m) \leq \probg(Y <  \eta m) \leq \exp \bigg( \frac{-\eta m}{4} \bigg),
$$
which completes the proof.
\end{proof}

\subsection{Survival of non-extinct hub stars between consecutive long one-phases}\label{subsec:isolation-hub-surv}

We next consider the survival of the infection in hub stars between consecutive long one-phases. We show an initially non-extinct hub star can survive independently of its surroundings for polynomial time with probability at least $O(1).$ Because the end of the long one-phase is a stopping time, the strong Markov property allows us to consider this starting from time $0$.

\begin{lemma}\label{lem:hubsurv}
Consider the process internal to a hub star, and let $L_t$ be the number of infected leaves at time $t$. Suppose $L_0 > 0$, and let $\tau_H = \inf\{t > 0: L_t = 0\}$. Then there exists $1 > \rho > 0$ and a constant $\kappa > 0$ such that

$$\prob(\tau_H > (2+2/\alpha)m^{\rho}) \geq \kappa.$$
\end{lemma}

\begin{proof}


For our fixed hub vertex, let $S^H_i$ be the time of the $i$th reinfection of the hub vertex, and let $T^H_i$ be the next time after $S_i^H$ at which the hub is no longer infected. Call the $i$th time interval $[S^H_i, T^H_i)$ the $i$th \textbf{infectious phase}, provided $S^H_i<\infty$. We say the $i$th infectious phase is \textbf{long} if $T^H_i - S^H_i>1/(1+\alpha)$. We first consider the probability that time $1$ is the start of a long infectious phase for this hub. We can guarantee this if either the hub is initially infected at time $0$ and
\begin{enumerate}
    \item [(C1)] There are no recovery dots or isolation crosses on the hub in $[0,1+1/(1+\alpha)]$,
\end{enumerate}
which has probability $e^{-(2+\alpha)}$; or a leaf is infected and
\begin{enumerate}
    \item [(C2)] There is a recovery dot on the hub in $[0,1/2]$,
    \item [(C3)] There are no recovery  or isolation crosses on the leaf in $[0,1]$,
    \item [(C4)] There is an infection arrow from the leaf to the hub in $[1/2,1]$,
    \item [(C5)] There are no recovery dots on the hub in $[1/2,1+1/(1+\alpha)]$, and
    \item [(C6)] There are no isolation crosses on the hub in $[0,1+1/(1+\alpha)]$
\end{enumerate}

The probability of satisfying (C2-C6) is
$$
K:=(1-e^{-1/2})(e^{-(1+\alpha)})(e^{-\lambda/2})(e^{-(1/2 + 1/(1+\alpha))})(e^{-\alpha(1 + 1/(1+\alpha))}) < e^{-(2+\alpha)}.
$$

We discount any hubs that fail to satisfy the above conditions and turn our attention to the hubs that have a long infectious phase starting from time $1$. Since the comparison process is attractive, we can simplify our argument by henceforth assuming that long infectious phases last for exactly $1/(1+\alpha)$ time. Broadly, the process internal to the hub star can survive between long one-phases if

\begin{enumerate}
    \item the hub infects a sufficient number of leaves during a long infectious phase, and
    \item a sufficient number of leaves stay infected until the start of the next long infectious phase.
\end{enumerate}

We consider 1. first. Suppose $T_i^H$ is the end of a long infectious phase and let $t = 1/(1+\alpha)$. By monotonicity we can set $T_i^H = S_i^H + t$ and consider the time interval $[S_i^H,S_i^H+t]$. For each leaf, we can ensure it is infected at time $S_i^H+t$ if the leaf has

\begin{enumerate}
    \item [(L1)] A recovery dot in $[S_i^H,S_i^H+t/3]$,
    \item [(L2)] An infection arrow from the hub in time $[S_i^H+t/3,S_i^H+t]$,
    \item [(L3)] No recovery dots in $[S_i^H+t/3, S_i^H+t]$, and
    \item [(L4)] No isolation crosses in $[S_i^H,S_i^H+t]$,
\end{enumerate}

The probability of satisfying (L1)-(L4) is

$$p^* = (e^{-t/3})(1-e^{-2\lambda t/3})(e^{-2t/3})(e^{-\alpha t}).$$
The number of infected leaves at the end of a long infectious phase stochastically dominates $X^*$, where
$$X^* \sim \textrm{Binomial}(m, p^*).$$

Let $D$ be the event that there are fewer than $mp^*/2$ infected leaves at time $T_i^H$. Then using a Chernoff bound,

\begin{equation}\label{eq:hub-long-one}
\prob(D) \leq \prob(X^* < mp^*/2) \leq \exp \bigg( \frac{-p^*m}{4}\bigg).
\end{equation}

We now consider the survival of these infections until the next long one-phase. Let $T_{\ell}$ be the time from $T_i^{H}$ until the next long one-phase. For $C > 0$ to be chosen later, define the following bad events.

\begin{enumerate}




    \item [(D1)] Consider the first $2C \log m$ time intervals after $T_i^H$ during which the hub is not in state $0$. Let $T^*_{\ell}$ be the total amount of this time, and let $D_1$ be the event $T^*_{\ell} > 4 C \log m$. Since $T_\ell^* \sim \textrm{Gamma}(2C \log m, 1)$, we have
    

    $$\prob(D_1) \leq \Big(\frac{2}{e}\Big)^{C \log m} = m^{-C(1- \log 2)}$$

    by a Chernoff bound.


    \item [(D2)] Consider the first $2C \log m$ time intervals after $T_i^H$ during which the hub is in state $0$. Let $T_{\ell}^{**}$ be the total amount of this time, and let $D_2$ be the event that $T_{\ell}^{**} >\frac{4 C \log m}{\alpha}$. Since $T_{\ell}^{**}$ is dominated by a $\textrm{Gamma}(2C \log m, \alpha)$ random variable, we have

    $$\prob(D_2) \leq \Big(\frac{2}{e}\Big)^{C \log m} = m^{-C(1- \log 2)}$$

    by a Chernoff bound.

    \item [(D3)] There is no leaf that is state $1$ at time $T_i^H$ and stays in state $1$ until time $T_i^H + 4C \log m + \frac{4C \log m}{\alpha}$. Call this event $D_3$.

    \item [(D4)] More than $\frac{2\lambda}{\lambda+\alpha}C \log m$ of the next $2 C \log m$ time intervals after $T_i^H$ during which the hub is in state $0$ are ended by transitions of the hub to state $-1$. Call this event $D_4$.

    \item [(D5)] None of the next $\frac{2\lambda}{\lambda+\alpha}C \log m$ infectious phases after time $T_i^H$ are long or there are fewer than $C \log m$ infectious phases after time $T_i^H$. Call this event $D_5$.

\end{enumerate}

    We observe on the event $D^c \cap D_1^c \cap D_2^c \cap D_3^c \cap D_4^c \cap D_5^c$ that $T_{\ell} \leq 4C \log m + \frac{4C \log m}{\alpha}$. We also observe

    \begin{equation*}
    \begin{aligned}
        (D^c \cap D_1^c \cap D_2^c \cap D_3^c \cap D_4^c \cap D_5^c)^c &= D \cup D_1 \cup D_2 \cup D_3 \cup D_4 \cup D_5\\
        &\subseteq D \cup D_1 \cup D_2 \cup (D_3 \cap D^c) \cup (D_4 \cap D_3^c \cap D_1^c \cap D_2^c) \\
        &\quad \cup (D_5 \cap D_4^c)
    \end{aligned}
    \end{equation*}

    First, $(D_3 \cap D^c)$ requires that there are at least $mp^*/2$ infected leaves at time $T_i^H$. The probability that a leaf that is in state $1$ at time $T_i^H$ does not leave state $1$ before time $T_i^H + 4C \log m + \frac{4C \log m}{\alpha}$ is $e^{-(4C \log m + \frac{4C \log m}{\alpha})(1+\alpha)}$, and this occurs independently for each leaf that is in state $1$ at time $T_i^H$. Thus,
    $$
    \prob(D_3 \cap D^c) \leq \exp \bigg( -(mp^*/4)(e^{-(4C \log m + \frac{4C \log m}{\alpha})(1 + \alpha)}\bigg).
    $$
    If we choose $C$ depending on $\lambda$ and $\alpha$ such that $e^{-(4C \log m + \frac{4C \log m}{\alpha})(1+\alpha)} = m^{-1/2}$, then,
    $$
    \prob(D_3 \cap D^c) \leq e^{-p^*\sqrt{m}/4}
    $$

    Next, consider $(D_4 \cap D_3^c \cap D_1^c \cap D_2^c)$. On the event $D_1^c \cap D_2^c$, there are at least $2 C \log m$ time intervals during which the hub is in state $0$ that occur before time $T_i^H + 4C \log m + \frac{4C \log m}{\alpha}$, and on the event $D_3^c$ there exists some leaf that is in state $1$ from time $T_i^H + t$ until time $T_i^H + t + 4C \log m + \frac{4C \log m}{\alpha}$. On the event $D_3^c \cap D_1^c \cap D_2^c$, there exists some leaf that is infected continuously through the next $2 C \log m$ time intervals during which the hub is in state $0$, so during each of these intervals there is probability at least $\frac{\lambda}{\alpha + \lambda}$ that the $0$ phase is ended by a transition to state $1$. Therefore, the number of $0$ phases ended by a transition of the hub to state $1$ dominates $Y^* \sim  \textrm{Binomial}(2 C \log m, \frac{\lambda}{\alpha + \lambda})$. Thus,
    \begin{align*}
    \prob(D_4 \cap D_3^c \cap D_1^c \cap D_2^c) &\leq P(Y^* \leq C \log m) \\
    &= P(Y^* \leq E(Y^*)/2) \\
    &\leq \exp\bigg( -\frac{\lambda C \log m}{2( \alpha + \lambda)}\bigg) \\
    &= m^{-\frac{\lambda C}{2(\alpha + \lambda)}},
    \end{align*}
    by a Chernoff bound. 

    Finally, consider the event $(D_5 \cap D_4^c)$. On $D_4^c$  there are at least $C \log m$ infectious phases after time $T_i^H$. Since each infectious phase is long with probability $e^{-1}$,
    $$
    \prob(D_5 \cap D_4^c) \leq \exp\bigg(-\frac{e^{-1}C \log m}{4} \bigg) = m^{- \frac{C}{4e}}
    $$

    We now compute
    \begin{equation}\label{eq:hubcycle}
    \begin{aligned}
        &\prob((D^c \cap D_1^c \cap D_2^c \cap D_3^c \cap D_4^c \cap D_5^c)^c) \\ 
        &\quad \leq \prob(D) + \prob(D_1) + \prob(D_2)  +\prob(D_3 \cap D^c)\\
        &\qquad + \prob(D_4 \cap D_3^c \cap D_1^c \cap D_2^c)  + \prob(D_5 \cap D_4^c)\\
        &\quad \leq 5 \max\bigg(\exp \bigg( \frac{-p^*m}{4}\bigg), m^{-C(1-\log 2)}, e^{-p^*\sqrt{m}/4}, m^{-\frac{\lambda C}{2(\alpha + \lambda)}}, m^{- \frac{C}{4e}}\bigg)\\
        &\quad =5 \max\bigg(m^{-C(1-\log 2)}, m^{-\frac{\lambda C}{2(\alpha + \lambda)}}, m^{- \frac{C}{4e}}\bigg)
    \end{aligned}
    \end{equation}
    for large enough $m$.

    If we consider cycles of the process internal to the hub moving between long infectious phases, then equation \eqref{eq:hubcycle} gives an upper bound on the probability that following a long infectious phase the hub star fails to survive long enough to have a subsequent long infectious phase. Now choose $1 > \rho > 0$ small enough such that 

    $$(2+2/\alpha)m^{\rho }(1+ \alpha) \leq (1/2)\bigg(5\max \bigg( m^{-C(1-\log 2)}, m^{-C \frac{\lambda}{\alpha + \lambda}}, m^{-\frac{5C}{108}}\bigg)\bigg)^{-1}.$$

Then the probability of seeing fewer than $(2+2/\alpha)m^{\rho }(1 + \alpha)$ successful cycles is at most $1/2$ for large enough $m$. Each cycle must last at least $1/(1+\alpha)$ time, since the initial long infectious phase lasts at least this much time, and so this guarantees 
$$
\prob(L_{t^*} > 0) \geq K/2 =:\kappa
$$
for $t^{*} = (2+2/\alpha)m^{\rho}$ as desired.
\end{proof}

\subsection{Time between long one-phases}\label{subsec:isolation-not-one-phase}


We now study the time between consecutive long one-phases. We begin with the following lemma.

\begin{lemma}\label{lemma:zero-prob}
    Suppose at time $0$ the process is in the zero-phase and there are $h$ non-extinct hub stars. Then the probability the center is infected at time $1$ with no other state changes to the center in the interim is at least $q\mathbf{1}_{h > 0}$ where
    $$
    q = (e^{-3\alpha})(e^{-8/3})(1-e^{-1/3})(e^{-2\lambda/3}).
    $$
\end{lemma}

\begin{proof}
    We first observe that if $h = 0$, the process is extinct and so $q\mathbf{1}_{h > 0} = 0$ as expected. If $h > 0$, then there exists some hub star where either the hub is infected at time $0$ or some leaf is infected at time $0$. In this case, the following sequence of marks in the graphical construction ensures the center is infected at time $1$.
        \begin{enumerate}
            \item[(E1)] None of the center, our chosen hub, or our chosen leaf have an avoidance cross in time $(0,1)$. This event has probability $e^{-3\alpha}$.
            \item[(E2)] Neither the center nor the leaf have a recovery dot in time $(0,1)$. This event has probability $e^{-2}$.
            \item[(E3)] The hub has a recovery dot in time $(0,1/3)$. This event has probability $1-e^{-1/3}$.
            \item[(E4)] The hub does not have a recovery dot in $(2/3,1)$. This event has probability $e^{-2/3}.$
            \item[(E5)] The leaf has an infection arrow to the hub in time $(1/3,2/3)$. This event has probability $1-e^{\lambda/3}$.
            \item[(E6)] The hub has an infection arrow to the center in time $(2/3,1)$. This event has probability $1-e^{-\lambda/3}$
        \end{enumerate}
        Since events (E1)-(E6) are independent in our graphical construction, we can multiply the probabilities and obtain $q\mathbf{1}_{h > 0} = (e^{-3\alpha})(e^{-8/3})(1-e^{-1/3})(e^{-2\lambda/3})$ as desired.
\end{proof}

Note that the strong Markov property implies that Lemma \ref{lemma:zero-prob} can also be applied starting from any stopping time $T^*$.



\begin{lemma}\label{lem:onereturn}
        Recall that $T_i$ is the next time after the $i$th reinfection of the center that the center is no longer infected. Let $W^{\textrm{long}}$
        be the amount of time from $T_i$ until the start of the next long one-phase, and let $N_t$ be the number of non-extinct hub stars at time $T_i+t$. Then given the configuration $\xi^R_{T_i}$ is such that $L_0\ge \eta m$ there exists $\rho > 0$ such that

        \[
        \prob^{\xi^R_{T_i}}\Big(W^{\textrm{long}} \le (2 + 2/\alpha)m^{\rho}\Big) \ge 1- 5e^{-Cm^{\rho}}
        \]
        where $\kappa$ is the constant from Lemma~\ref{lem:hubsurv}, $C$ is a constant depending on $\alpha$ and $\lambda$, and $\prob^{\xi^R_{T_i}}$ means probability conditioned on $\xi^R_{T_i}$.
\end{lemma}

\begin{proof}
Let $W^{\textrm{long}}$ be the amount of time from $T_i$ until the next long one-phase. Observe that in the comparison process, every vertex switches infinitely many times between state $0$ and states $\{-1,1\}$; in particular, this is true of the center vertex. For any $\rho > 0$, define the following bad events. 

\begin{enumerate}


    \item[(B1)] Consider the first $m^{\rho}$ time intervals after time $T_i$ during which the center is in states $\{-1,1\}$. Let $T_g^*$ be the total amount of this time, and let $B_1$ be the event that $T_g^* >2 m^{\rho}$. Since $T_g^* \sim \textrm{Gamma}(m^{\rho}, 1)$ for every possible $\xi^R_{T_i}$, we observe

    $${\prob^{\xi^R_{T_i}}}(B_1) \leq \Big(\frac{2}{e}\Big)^{m^{\rho}} = e^{-m^{-\rho}(1- \log (2))}$$

    by the Chernoff bound.

    \item[(B2)] Consider the first $m^{\rho}$ time intervals after time $T_i$ during which the center is in state $0$. Let $T_g^{**}$ be the total amount of this time, and let $B_2$ be the event that $T_g^{**} > \frac{2 m^{\rho}}{\alpha}$. Since $T_g^{**}$ is dominated by a  $\textrm{Gamma}(m^{\rho}, \alpha)$ random variable for every possible $\xi^R_{T_i}$, we observe

    $${\prob^{\xi^R_{T_i}}}(B_2) \leq \Big(\frac{2}{e}\Big)^{m^{\rho}} = e^{-m^{-\rho}(1- \log (2))}$$

    by the Chernoff bound.

    \item[(B3)] Let $N_t$ denote the number of hub stars with at least one infected vertex at time $t$. Let $B_3$ be the event that $L_{t} \le \eta\kappa m / 2$ for some $T_i < t \leq T_i +(2+2/\alpha)m^{\rho}$, where $\kappa$ is the constant from Lemma~\ref{lem:hubsurv}. To compute the probability of $B_3$, we observe that by Lemma \ref{lem:hubsurv}, the number of hub stars that have at least one infected vertex for all times in $[T_i,T_i +(2+2/\alpha)m^{\rho}]$ dominates $X \sim \textrm{Binomial}(\eta m,\kappa)$, since $L_0\ge \eta m$ by our assumption on $\xi^R_{T_i}$. Then
    $$
    \prob^{\xi^R_{T_i}}(B_3) \leq \exp(-\eta\kappa m / 8).
    $$

    \item[(B4)] Let $B_4$ be the event that fewer than $(q/2)m^{\rho}$ of the next $m^{\rho}$ times after $T_i$ at which the center changes from state $0$ to a state in $\{-1,1\}$ result in a one-phase (a change to state $1$). Each time the center returns to state $0$, it is in state $1$ after one unit of time with no other changes of the state of the center in the interim with probability at least $q \mathbf{1}_{\{N_t>0 \text{ for all } T_i<t<T_i+(2+2/\alpha)m^\rho\}}$ by Lemma \ref{lemma:zero-prob}. On the event $B_1^c\cap B_2^c\cap B_3^c$ we have that $h_t>0$ for all $T_i < t \leq T_i +(2+2/\alpha)m^{\rho}$ and that the next $m^{\rho}$ times the center changes from state $0$ to states $\{-1,1\}$ all occur before time $T_i +(2+2/\alpha)m^{\rho}$. Thus, on this event, the number of the next $m^{\rho}$ time intervals after $T_i$ during which the center is in state $\{-1,1\}$ that start in state $1$ dominates  $X \sim \textrm{Binomial}(m^{\rho},q)$ and we have
$$
\prob^{\xi^R_{T_i}}(B_4 \cap B_1^c \cap B_2^c \cap B_3^c) \leq \exp\bigg(-\frac{q m^{\rho}}{8} \bigg).
$$
    \item[(B5)] Let $B_5$ be the event that none of the first $(q/2)m^{\rho}$ one-phases after time $T_i$ is long.
    



\end{enumerate}

We observe on the event $B_1^c \cap B_2^c \cap B_3^c \cap B_4^c \cap B_5^c$ that $W^{long} \leq (2 + 2/\alpha)m^{\rho}$. We also observe

\begin{equation*}
\begin{aligned}
(B_1^c \cap B_2^c \cap B_3^c \cap B_4^c \cap B_5^c)^c &= B_1 \cup B_2 \cup B_3 \cup B_4 \cup B_5\\ 
&\subseteq B_1 \cup B_2 \cup B_3 \cup (B_4 \cap B_1^c \cap B_2^c \cap B_3^c) \cup B_5.
\end{aligned}
\end{equation*}

We can now compute
\begin{equation*}
\begin{aligned}
\prob^{\xi^R_{T_i}}((B_1^c \cap B_2^c \cap B_3^c \cap B_4^c \cap B_5^c)^c) &\leq \prob^{\xi^R_{T_i}}(B_1) + \prob^{\xi^R_{T_i}}(B_2) + \prob^{\xi^R_{T_i}}(B_3)\\
&+ \prob^{\xi^R_{T_i}}(B_4 \cap B_1^c \cap B_2^c \cap B_3^c) + \prob^{\xi^R_{T_i}}(B_5 \cap B_4^c)\\
&\leq 5\max \bigg(e^{-m^{\rho}(1- \log(2))},(1/2)^{\eta m},\exp\bigg(-\frac{q m^{\rho}}{8}\bigg) \bigg)\\
&\leq 5e^{-Cm^{\rho}}
\end{aligned}
\end{equation*}
for some constant $C$, completing the proof.

\end{proof}

\subsection{Main result for power law random graphs}\label{subsec:isolation-main-powerlaw}

To prove Theorem~\ref{thm:powerlaw}, we will first show that the configuration model random graph with power law degree distribution contains many stars of stars of order $m = n^{\epsilon}$ for some $\epsilon > 0$ with high probability.

\begin{lemma}\label{lem:power-law-star}
    Let $G = (V,E)$ be a graph with $n$ vertices generated according to a configuration model with power law degree distribution with power law exponent $\gamma > 3$ and minimum degree $3$. For any $\nu>0$ there exists an $\epsilon>0$ such that with high probability, $G$ contains at least $n^{1-\nu}$ disjoint subgraphs isomorphic to a star of stars of order $m=n^{\epsilon}$.
\end{lemma}
\begin{proof}
We will use a first-moment argument to identify a large number of vertices that are centers of stars of stars of order $m = n^{\epsilon}$ with $\epsilon>0$ chosen sufficiently small by the end of the proof. Let $\delta>0$ be small, ultimately depending on $\epsilon$ and $\gamma$.  We first show that the probability that vertex $1$ is the center of a star of stars of order $m$ conditional on the event that $D_1\in [n^{3\gamma\epsilon}, n^{3 \gamma\epsilon + \delta}]$ tends to $1$ as $n\to\infty$. Until further notice, we condition on the value of $D_1$ and assume that it is between $n^{3\gamma\epsilon}$ and $n^{3 \gamma\epsilon + \delta}$, and by the law of large number, we also assume the total number of half-edges is $S = \sum_{i\in V} D_i \sim E(D_i)n  = \mu n$.

By the law of large numbers, with high probability there are at least $cn^{1-2\epsilon(\gamma-1)}$ vertices among $\{2, \ldots, n/2\}$ with degrees $D_i\ge n^{2\epsilon}$. Let $V_1$ denote the first $cn^{1-2\epsilon(\gamma-1)}$ such vertices. We now consider how many vertices among $V_1$ are connected to vertex $1$, and to do so we attempt to pair only the first $n^{\epsilon}$ half edges from each vertex of $V_1$ to the half-edges of vertex $1$.
Let $E_1$ be the set of edges that get formed by pairing one of the first $n^{3\gamma\epsilon}$ half-edges of vertex $1$ to one of the first $n^{\epsilon}$ half-edges of any vertex in $V_1$ --- there are $n^{1-2\gamma\epsilon-\epsilon}$ such half-edges. The number of edges formed in this way, $|E_1|$, dominates a Hypergeometric random variable with a population comprised of $n^{1-2\gamma \epsilon-\epsilon}$ ``successes" and $S-n^{1-2\gamma \epsilon-\epsilon}\sim \mu n$ ``failures'' and sample size $n^{3\gamma\epsilon}$. It follows that $|E_1|\ge \mu^{-1}n^{(\gamma-1)\epsilon}(1-o(1))$ with high probability.

Now, let $N_1$ be the set of vertices in $V_1$ with at least one of its first $n^{\epsilon}$ half-edges paired to one of the first $n^{3\gamma\epsilon}$ half-edges of vertex $1$. Given that $|E_1| =\mu^{-1} n^{(\gamma-1)\epsilon}(1-o(1))$, the paired half-edges are uniformly distributed on all possible pairings of this size. We can therefore select this pairing by first choosing $|E_1|$ half-edges connected to vertex $1$, then choosing the $|E_1|$ half-edges connected to vertices of $V_1$, then matching them. By selecting the half-edges of $E_1$ that are connected to $V_1$ sequentially, we see that until we have selected $|V_1|/2\gg |E_1|$ half-edges connected to $V_1$, at each step we have at least a $1/2$ probability of selecting a half-edge connected to a vertex of $V_1$ that has not had any half-edges selected previously. It follows that $|N_1|\ge n^{(\gamma-1)\epsilon}/(4\mu)$ with high probability.

For each of the first $n^{2\epsilon} \le |N_1|$ vertices in $N_1$, we now wish to find at least $n^{\epsilon}$ unique neighbors among the vertices $\{n/2+1, \ldots, n\}$. Let $V_2$ be the set of vertices among $\{n/2+1, \ldots, n\}$ that have at least one unpaired half-edge at this point in the process, identify this half-edge with the corresponding vertex in $V_2$, and note that $|V_2|\ge n/2 - n^{1-2\gamma\epsilon-\epsilon}-n^{3\gamma\epsilon} \ge n/4$.  Each vertex of $N_1$ has degree at least $n^{2\epsilon}$, and we have observed only $n^{\epsilon}$ of its half edges. Thus, for each vertex of $N_1$, we now consider the next $n^{2\epsilon}/2$ of its half-edges that haven't been paired yet. We call a vertex of $N_1$ successful if at least $n^{\epsilon}$ of its half-edges are paired with half edges of $V_2$, and we check sequentially whether each vertex of $N_1$ is successful. After having paired the first $k$ half-edges attached to vertices of $N_1$ in this procedure, the next half-edge attached to a vertex in $N_1$ gets paired to a half-edge of $V_2$ with probability at least $(n/4 - k) / (\mu n)\ge 1/8\mu$ while $k\le n/8$. As the total number of half-edges under consideration is at most $n^{4\epsilon}\le n/8$, we have that each half-edge attached to $N_1$ has at least probability $1/8\mu$ to connect to $V_2$. Therefore, probability that a fixed vertex among the first $n^{2\epsilon}$ of $N_1$ is successful is at least the probability that a Binomial$(n^{2\epsilon}/2, 1/8\mu)$ random variable exceeds $n^{\epsilon}$, which tends to $1$. Therefore, the number of successful vertices in $N_1$ exceeds $n^{\epsilon}$ with high probability. We have thus shown that the conditional probability that vertex $1$ is the center of a star of stars of order $m=n^{\epsilon}$ given that $D_1\in [n^{3\gamma\epsilon}, n^{3\gamma\epsilon+\delta}]$ is $1-o(1)$. 

Let $H = \{j\in V : D_j \in [n^{3\gamma\epsilon}, n^{3\gamma\epsilon + \delta}]\}.$ Then $|H| \asymp n^{1-3\gamma(\gamma-1)\epsilon}$ with high probability. Let $B_4(i)$ denote the set of vertices within distance $4$ of $i$ in $G$. For $j\in H$, let $H_j = H\setminus\{j\}$ and let $A_j$ be the event that $j$ is the center of a star of stars of order $m=n^{\epsilon}$ and that $j\notin \bigcup_{i\in H_j} B_4(i)$. Then,
\begin{align*}
    \prob(A_j^c) &\le \prob(j\text{ is not a center}) + \prob\left(\left|\bigcup_{i\in H_j} B_4(i) \right| \ge n^{3\gamma\epsilon+2\delta} |H|\right)\\
    &\quad + \prob\left(j\in \bigcup_{i\in H_j} B_4(i), \left|\bigcup_{i\in H_j} B_4(i) \right| < n^{3\gamma\epsilon+2\delta} |H|\right).
\end{align*}
We claim that all three terms on the right side tend to 0. The first term tends to $0$ by the argument given above. The second term tends to $0$ by the law of large numbers, as the size-biased degree distribution has finite mean. The third term tends to $0$ because if $\left|\bigcup_{i\in H_j} B_4(i) \right| < n^{3\gamma\epsilon+2\delta} |H| \le Cn^{1+3\gamma\epsilon +2\delta -3\gamma(\gamma-1)\epsilon}$, the probability that a half-edge connected to $j$ was paired with one of the half-edges within $\bigcup_{i\in H_j} B_4(i)$ is at most $Cn^{1+3\gamma\epsilon +2\delta -3\gamma(\gamma-1)\epsilon}\cdot n^{3\gamma\epsilon+\delta}/n = Cn^{6\gamma\epsilon +3 \delta- 3\gamma(\gamma-1)\epsilon} \to 0$ as $n\to\infty$ for sufficiently small $\delta$, since $\gamma>3$.
Then by Markov's inequality,
\begin{equation}
    \prob\left(\sum_{j\in H} 1_{A_j^c} \ge |H|/2\right)  = o(1),
\end{equation}
so with high probability the number of disjoint stars of stars of order $m=n^{\epsilon}$ with centers in $H$ is at least $|H|/2 \asymp n^{1-3\gamma(\gamma-1)\epsilon}$. Choosing $\epsilon$ sufficiently small completes the proof.
\end{proof}

We are now ready to prove Theorem \ref{thm:powerlaw}. Armed with the key ingredients of Lemmas~\ref{lemma:one-phase-lit}, \ref{lem:onereturn} and~\ref{lem:power-law-star}, the proof of the analogue of Theorem~\ref{thm:powerlaw} for the comparison process follows by a small modification to the proof that the usual contact process survives for (stretched) exponentially long on power-law random graphs, due to Chatterjee and Durrett~\cite{ChatterjeeDurrett2009}. Thus, we sketch the overview of the proof and omit some of the details.

\begin{proof}[Proof of Theorem~\ref{thm:powerlaw}]

By Lemma~\ref{lem:power-law-star}, for any $\nu>0$ there is an $\epsilon>0$ such that there are (at least) $N=n^{1-\nu}$ disjoint stars of stars of order $m = n^{\epsilon}$; denote these subgraphs by $S_1, \ldots, S_N$. Also note that  the diameter of $G$ is at most $C\log n$ with high probability by Theorem 7.19 of~\cite{vdH2}. By Lemma~\ref{lem:Remenik comparison} it suffices to prove that the comparison process survives for time $e^{n^{1-2\nu}}$ with high probability.

Let $\beta>0$ and let $t_k = ke^{m^\beta}  = k\exp({n^{\epsilon\beta}})$ for $k = 0, 1, \ldots$. We say the star of stars $S_i$ begins a \textbf{cycle} at time $t$ if at this time its center ends a long one-phase (has been infected for time $1/(1+\alpha)$) and has at least $\eta n$ non-extinct hub stars, where $\eta$ is the constant from Lemma~\ref{lemma:one-phase-lit}. We say that this cycle is \textbf{successful} if another cycle begins within time $(2+2/\alpha)m^\rho$, where $\rho$ is the constant from Lemma~\ref{lem:onereturn}.

For $k\ge 1$, we call $S_i$ \textbf{hot} at time $t_k$ if it has at least one cycle that begins in $[t_{k-1},t_k]$ and every one of its cycles that begin in $[t_{k-1},t_k]$ is successful. 
Since each cycle must include a long one-phase of the center, the maximum number of cycles during $[t_k, t_{k+1}]$ is less than $e^{m^\beta}$. Therefore, if $S_i$ is hot at time $t_k$, then by Lemmas~\ref{lemma:one-phase-lit} and~\ref{lem:onereturn}, the probability that $S_i$ will not be hot again at time $t_{k+1}$ is at most $e^{m^\beta}(e^{-\eta m/4} + 5e^{-Cm^{\rho}}) \le e^{-m^\beta}$ for small enough $\beta$. Also, observe that since these lemmas concern the dynamics internal to $S_i$, if $S_i$ is hot at time $t_k$, then it will (with the same probability) contain an infected vertex throughout the interval $[t_k, t_{k+1}]$.

If $S_i$ is hot at time $t_k$ and $S_j$ is another star of stars, then as above $S_i$ sustains the infection internally through time $t_{k+1}$ and is hot again at time $t_{k+1}$ with probability at least $1-e^{-m^{\beta}}$. Now subdivide the time interval $[t_k, t_{k+1}]$ into subintervals of length $C\log n$, and during each subinterval of time, we attempt to ``push'' the infection from one of the infected vertices in $S_i$ along a fixed shortest-path (of length at most $C\log n$) to the center of $S_j$. This can be done in a greedy fashion: the infection is pushed from vertex $v_\ell$ to $v_{\ell+1}$ along this path if in a unit time interval, $v_\ell$ does not recover or isolate during the entire interval, $v_{\ell+1}$ recovers and/or returns from isolation exactly once, and $v_{\ell}$ infects $v_{\ell+1}$ after it is susceptible. Moreover, for a push to be considered a success, we require that the center of $S_j$ enters a long one-phase when it becomes infected at the end of this path. These events each occur with positive probability, and so the probability that one such push results in a long one-phase of $S_j$ is at least $e^{-C\log n} = n^{-C}$, for some $C>0$. As we make at least $e^{m^{\beta}/2}$ attempts to push the infection from $S_i$ to $S_j$, $S_j$ will complete a long one-phase during $[t_k, t_{k+1}]$, and, by Lemmas~\ref{lemma:one-phase-lit} and~\ref{lem:onereturn}, will therefore be hot at time $t_{k+1}$ with probability at least $1- e^{-m^{\beta}/4}$. By a similar argument, since all vertices are initially infected, the center of $S_i$ has a positive probability of having a long one-phase starting at time $0$, so with high probability there are at least $\beta N$ hot stars of stars at time $t_1$.

We define $X_k$ to be the number of stars of stars, $S_1, \ldots, S_N$, that are hot at time $t_k$, and we claim that while $X_k\in [1,N-1]$, the process $X_k$ dominates a random walk process with positive drift. 
Indeed, since $S_1, \ldots, S_N$ are disjoint and the comparison process is attractive, if $X_k \ge 1$, then the number of stars of stars that are hot at time $t_k$ and are not hot at time $t_{k+1}$ is dominated by $Z\sim \text{Binomial}(N, e^{-m^{\beta}})$. If $X_k\in[1, N-1]$ we can identify an $i$ and $j$ such that $S_i$ is hot at time $t_k$ but $S_j$ is not hot at time $t_k$, and the probability that all hot stars of stars at time $t_k$ remain hot at time $t_{k+1}$ and that pushing the infection from $S_i$ to $S_j$ makes $S_j$ hot at time $t_{k+1}$ is at least $2/3$. Summarizing, the transition probabilities for $X_k\in [1,N-1]$ are bounded as
\begin{align*}
\prob(X_{k+1}-X_k \ge 1 \mid X_k\in [1,N-1]) &\ge 2/3\\
\prob(X_{k+1}-X_k  \le -z \mid X_k\in [1,N-1])  &\le \prob(\mathrm{Binomial}(N, e^{-m^{\beta}}) \ge z) \qquad \text{for } z\ge 1.
\end{align*}

Therefore,  $X_k$ dominates a random walk on $\mathbb{Z}$ with increments distributed as $Y - Z$ where $Y\sim$ Bernoulli(2/3) is independent of $Z$, until this random walk exits $[1,N-1]$. The moment generating function for an increment of this random walk is 
$$
\psi(s) = \frac{2e^s-1}{3} [1 +(e^{-s}-1)e^{-m^\epsilon}]^N.
$$
Observe that $Ne^{-m^\epsilon}\to 0$ as $n\to\infty$, so $\psi(s)\to (2e^{s}-1)/3$, and for large $n$ we have $\psi(-1)<1$. Denoting this random walk by $\hat X_k$, it follows that $M_k = \exp(-\hat X_k)$ is a nonnegative supermartingale. Therefore, if $X_1\ge \beta N$ and $K = \inf\{k: \hat X_k\le 0\}$, then $\prob(K<\infty)\le \E(M_K) \le EM_0 \le e^{-\beta N} = e^{-\beta n^{1-\eta}}.$ 

Now we compare $X_k$ to $\hat X_k$ each time $X_k$ hits $N$ then jumps to some value less that $N$. The probability that $X_{k+1}< N/2$ given that $X_{k}=N$ is at most $e^{-N/3}$, so with high probability no such jumps are made during the first $e^{N/4}$ steps of $X_k$. Thus, if $X_k$ hits $N$, we wait for it to reenter $[N/2, N-1]$. Once it does, it dominates $\hat X_k$, starting from the same value $\ge N/2$, until the first time $\hat X_k$ exits $[1,N-1]$. The probability that $\hat X_k$ exits this interval on the left is at most $e^{-N/2}$, and so it will eventually exit on the right, implying $X_k$ returns to $N$ before hitting $0$. This implies $X_k\ge 1$ for $k\le e^{N/4}$ with high probability, and the infection survives for at least time $t_{e^{N/4}}\ge e^{n^{1-2\nu}}.$
%
\end{proof}

\section{Vigilance model}\label{sec:vigilance}

\subsection{Overview}

Our goal for the vigialnce model is to show that for graphs satisfying Assumption \ref{assumption:isoperimetric}, for every fixed $\alpha > 0$ there is a phase transition in $\lambda$ between linear and exponential asymptotic survival regimes. We will study the size of the infected set over time, and so introduce the following notation. Let $S_t$, $I_t$, and $A_t$ denotes the sets of healthy, infected, and isolated vertices at time $t$ respectively, and let $|S_t|$, $|I_t|$, and $|A_t|$ denote the sizes of these sets.

It is straightforward to show that the infection dies out quickly when $\alpha > \lambda$. If we consider a single $SI$ pair, then the $I$ will become isolated before infecting the $S$ with probability $\frac{\alpha}{\alpha + \lambda} > 1/2$. This suggests that when $\alpha > \lambda$, $|I_t|$ will have negative drift, and we can use a standard random walk comparison to show the process dies out in linear time.

When $\lambda$ is much large than $\max(\alpha,1)$, a single $SI$ pair is more likely to result in an infection of the $S$ vertex rather than an isolation or recovery of the $I$ vertex. This suggests it is possible that $|I_t|$ has positive drift. However, there are two complications. First, we must consider that $|A_t|$ could be large. Second, $|\partial I_t|$ could be small, and combined with large $|A_t|$, lead to a configuration in which the number of $SI$ pairs is smaller than $\lambda|I_t|$. If this is the case, then $|I_t|$ will have negative drift, as the next recovery occurs at rate $|I_t|$ while the next infection occurs at rate $\lambda$ times the number of $SI$ pairs at time $t$.

To show such configurations are not typical, we prove Lemma \ref{lemma:a_time} and use Assumption \ref{assumption:isoperimetric}. Assumption \ref{assumption:isoperimetric} guarantees that when $|I_t|$ is not too large compared to the size of the graph $|V|$, $|\partial I_t|$ is not too small compared to $|I_t|$. Lemma \ref{lemma:a_time} shows that if we choose $\lambda$ sufficiently large compared to $\alpha$, then we can obtain a uniform bound on $|A_t|$ in terms of $|I_t|$ that holds for all $t \leq e^{Cn}$ with high probability, where the constant $C$ depends on $\alpha$ and $\lambda$ as well as the constants $\epsilon$ and $\delta$ in Assumption \ref{assumption:isoperimetric}.

We can then construct a renewal argument for the survival of the process. We start from all vertices infected and show in Lemma \ref{lemma:renewal} that there is some small $\epsilon^* > 0$ such that whenever $|I_t|$ drops below $(2/3)\epsilon^*n$, so long as our uniform bound on $|A_t|$ holds, the probability that $|I_t|$ hits $(1/3)\epsilon^*n$ before it hits $\epsilon^*n$ is $e^{-Cn}$ for some constant $C$ that depends on $\alpha$ and $\lambda$ as well as the constants $\epsilon$ and $\delta$ in Assumption \ref{assumption:isoperimetric}. Thus, we expect to observe $O(e^{-Cn})$ successful cycles of this renewal, which ensures the infection survives so long as no cycle has failed.

\subsection{Isoperimetric assumption}

We study the vigilance model on graphs satisfying isoperimetric condition in Assumption \ref{assumption:isoperimetric}. Before studying the vigilance model such graphs, we show that a configuration model power law random graph with minimum degree $3$ satisfies this assumption with high probability.

\begin{theorem}\label{thm:iso}
    Let $G = (V,E)$ be a configuration model graph with $n$ vertices and degree sequence $D_1, \ldots, D_n$ sampled iid such that
    $$
    \prob(D_i \ge k) \asymp \frac{1}{k^{\gamma-1}},\quad k = 3,4,\ldots
    $$
    where $\gamma > 3$. Then $G$ satisfies Assumption \ref{assumption:isoperimetric} with probability $1 - o(1)$ as $n \rightarrow \infty$.
\end{theorem}

\begin{proof}
    Observe that $\E(D_i) < \infty$ and $\mathrm{Var}(D_i) < \infty$ by our choice of $\gamma > 3$. Define $\bar{D} = \frac{1}{n}\sum_{i =1}^n D_i$. By the Weak Law of Large Numbers, there exists a constant $b_0$ such that
    $$
    \prob(3 < \bar{D} < b_0) \to 1 \qquad \text{as $n\to\infty$}
    $$
    and call this event $A_0$.

    Suppose $1/2 > \epsilon > 0$ and suppose we arbitrarily choose a set $U$ of $\epsilon n$ vertices. Then if $S(U)$ is the sum of the number of half edges of these vertices,
    $$
    3\epsilon n \leq S(U) \leq \sum_{i = n-\epsilon n}^n D_{(i)}
    $$
    where $D_{(i)}$ is the $i$th order statistic of $\{D_1,\ldots,D_n\}$. To bound $\sum_{i = n-\epsilon n}^n D_{(i)}$, we use the results in \cite{OrderStat}, which provide a central limit theorem for trimmed means that can be applied to the special case of the upper tail sum. Following their notation, we let $F$ be the cdf of the degree distribution, $q_\epsilon = \sup\{x: F(x) \leq 1- \epsilon\}$, and define 
    \[
    F_{\textrm{trimmed}}(k) =
    \begin{cases}
    0, & \text{if } k \leq q_\epsilon \\
    \frac{F(k) -(1- \epsilon)}{\epsilon}, & \text{if } k > q_\epsilon.
    \end{cases}
    \]
    
    When the variance of $F_{\textrm{trimmed}}$ is finite, which our choice of $\gamma > 3$ guarantees, the results of \cite{OrderStat} imply that $S(U)/n \overset{P}{\rightarrow} \mu_t$ where $\mu_t$ is the mean of $F_{\textrm{trimmed}}$. Observe $\mu_t = E(D|D \geq q_\epsilon)$. We see
    $$
    \prob(D > k) = \sum_{\ell= k+1}^{\infty}\frac{C_1}{\ell^{\gamma}} \sim \int_{k}^{\infty}\frac{C_1}{x^\gamma}dx=\frac{C_1 k^{1-\gamma}}{(\gamma - 1)},
    $$
    and setting $\prob(D > k) = \epsilon$, we get that $q_\epsilon \propto \epsilon^{-1/(\gamma - 1)}$. Thus
    $$
    \E(D|D \geq q_\epsilon) = \frac{\sum_{k=q_\epsilon}^{\infty}k^{-\gamma + 1}}{\sum_{k=q_\epsilon}^{\infty}k^{-\gamma}} \sim \frac{\int_{q_\epsilon}^{\infty} k^{-\gamma + 1}dk}{\int_{q_\epsilon}^{\infty}k^{-\gamma}dk} = \frac{\gamma - 1}{\gamma - 2}q_\epsilon.
    $$
    Thus we have
    $$
     \epsilon n \E(D|D \geq q_\epsilon) = C_2 n \epsilon^{(\gamma -2)/(\gamma -1)}(1 + o(1))
    $$
    and so we can choose $k_{\epsilon} = C\epsilon^{(\gamma - 2)/(\gamma - 1)}$ for some constant $C>C_2$ such that
    $$
    \prob(S(U) \leq k_{\epsilon}n)\rightarrow 1 \textrm{ as }n \rightarrow \infty.
    $$
    Let $A_1$ be the event that $S(U) \leq k_{\epsilon}n$. Then $\prob(A_1) = 1 - o(1)$.

    We now count the external vertex boundary of $U$. Order the half edges in $U$, $u_1, \ldots, u_{S(U)}$. Beginning from $u_1$, we pair each half edge of $U$ to another half edge of $V$. If a half edge $u_j$ has already been matched with another half edge $u_i$ for $i<j$, then we skip the half edge $u_j$. We can skip over at most $S(U)/2$ half edges, and so we are guaranteed at least $S(U)/2$ matchings. When we match each non-skipped half edge $u_i$, there are three possible outcomes:
    \begin{enumerate}
        \item $u_i$ matches to $u_j$ for some $j > i$,
        \item $u_i$ matches to the half edge of some vertex $v \notin U$ that has a half edge which is already matched to another half edge $u_{\ell}$ with $\ell < i$, or
        \item $u_i$ matches to the half edge at some vertex $v \notin U$ such that no half edge $u_{\ell}$ with $\ell < i$ has previously matched with a half edge at $v$.
    \end{enumerate}


    Whenever we observe outcome 3., we increase our count of the external vertex boundary of $U$ by $1$. The probability of observing outcomes 1. or 2. is at most $2S(U)/(\bar{D}n - 2S(U))$. If we choose $\epsilon$ sufficiently small, then on the event $A_0 \cap A_1$ we have $\bar{D} - 2S(U) \geq 3-2k_\epsilon>2$.  Thus, if we let $X$ count the number of non-skipped half edges of vertices in $U$ that result in outcomes 1. or 2., then given $A_0\cap A_1$, if $S(U) = s n$ with $s\le k_\epsilon$, we have $X\preceq Y + sn/2$ where
    
    $$
    Y \sim \textrm{Binomial}(sn/2, s),
    $$
    and $\mu = \E(Y) = \frac{s^2 n}{\bar{D} - 2s}$. 
    
    Let $k>1$ (to be chosen later) and let $A_2$ be the event $Y \leq \frac{sn}{2k}$. Using a Chernoff bound \cite{chernoff} 
    
    $$
    \prob\bigg(Y \leq \frac{sn}{2k}\bigg\lvert A_0 \cap A_1\bigg) \geq 1 - \exp\bigg(-\frac{sn}{2k}\bigg(\log \frac{.9\bar{D}}{ks}-1\bigg)+\frac{s^2n}{.9\bar{D}}\bigg).
    $$
    Let $k' > k$, to be chosen later. Then for sufficiently small $\epsilon > 0$ and $s\le k_\epsilon$,

    \begin{equation}\label{eq:k}
    1 - \exp\bigg(-\frac{sn}{2k}\bigg(\log \frac{.9\bar{D}}{k} + \log(1/s)-1\bigg)+\frac{s^2n}{.9\bar{D}}\bigg) \geq 1 - \exp\bigg(-\frac{sn}{2k'}\log(1/s)\bigg).
    \end{equation}
    Since we chose $k > 1$, on the event $Y \leq \frac{sn}{2k}$, the external vertex boundary of $U$ is at least size
    $$
    sn - \bigg(sn/2 + \frac{sn}{2k}\bigg) = \frac{(k-1)sn}{2k}.
    $$
 Thus, on the event $A_0 \cap A_1 \cap A_2$, the set of vertices $U$ has external vertex boundary at least $\frac{(k-1)sn}{2k} \ge \frac{3(k-1)}{2k}|U|$, since the minimum degree is $3$ so that $S(U) \ge 3|U|$.

    Now consider that the number of possible sets of vertices of size $\epsilon n$ is ${n \choose \epsilon n}$. Using Stirling's approximation

    \begin{equation}\label{eq:stirling}
    {n \choose \epsilon n} \leq \exp\bigg(1.1 \epsilon \log(1/\epsilon)n \bigg)
    \end{equation}
    for all sufficiently large $n$ when $\epsilon$ is sufficiently small. On the event $A_1$, $S(U) = s$ for some $3\epsilon n \leq s \leq k_{\epsilon}n$, and the preceding estimates hold uniformly for all sets $U$ and numbers of half edges $S(U)$ in this range. Let $A_2^i$ be the event that $A_2$ holds for set $U^i$ where $i = 1, \ldots {n \choose \epsilon n}$ enumerates all possible sets of size $\epsilon n$. Using a union bound and the fact that $s\log(1/s)$ is increasing for $s\in(0,1/e)$, by choosing $\epsilon$ so that $k_\epsilon<1/e$ we have
    $$
    \prob(A_0 \cap A_1 \cap A_2^1 \cap \ldots \cap A_2^{{n \choose \epsilon n}}) \geq 1 - o(1) - \exp\bigg(-\frac{k_\epsilon n}{2k'}\log(1/k_\epsilon )+1.1 \epsilon \log(1/\epsilon)n\bigg).
    $$
Noting that $k_\epsilon= C\epsilon^{(\gamma - 2)/(\gamma - 1)} \gg \epsilon$, we can choose $\epsilon$ small enough such that $
\frac{k_\epsilon}{2k'}\log (1/k_\epsilon) - 1.1 \epsilon\log(1/\epsilon) >\epsilon'>0$, so that

    $$
    1 - \exp\bigg(-\frac{k_\epsilon n}{2k'}\log(1/k_\epsilon)+1.1 \epsilon \log(1/\epsilon)n\bigg) \geq 1 - e^{-\epsilon' n}.
    $$
    \end{proof}

\subsection{The subcritical regime}\label{subsec:vigilance-subcritical}
The following lemma establishes the existence of a subcritical regime when $\alpha < \lambda$.

\begin{lemma}
Suppose we have any graph $G = (V,E)$ and rates $\lambda, \alpha > 0$ such that $\alpha > \lambda$. Then the epidemic dies out quickly.
\end{lemma}

\begin{proof}
Recall $|I_t|$ is the count of infected vertices at time $t$. We will show that $|I_t|$ is stochastically dominated by a random walk that is biased toward $0$.

Starting at any time $t$ there are four possibilities for the next event: infection, recovery, isolation, and return from isolation. Infection increases $|I_t|$ by $1$, recovery and isolation decrease $|I_t|$ by $1$, and return from isolation does not change $|I_t|$ but increases the number of $SI$ edges in the system. We can observe

\begin{enumerate}
    \item The rate of the next infection is $\lambda (\# SI \textrm{ edges})$
    \item The rate of the next recovery event is $|I_t|$
    \item The rate of the next isolation event is $\alpha (\# SI \textrm{ edges})$
\end{enumerate}

The the ratio of the rate of events that increase $|I_t|$ by $1$ to the rate of all events that change $|I_t|$ is

$$\frac{\lambda (\# SI \textrm{ edges})}{\lambda (\# SI \textrm{ edges}) + |I_t| + \alpha(\# SI \textrm{ edges})} \leq \frac{\lambda}{\lambda + \alpha}$$

where the right-hand side does not depend on the number of $SI$ edges. Thus, the probability that the next event that changes $|I_t|$ increases $|I_t|$ by $1$ is at most $\frac{\lambda}{\lambda + \alpha}$ and the probability that the next event that changes $|I_t|$ decreases $|I_t|$ by $1$ is at least $\frac{\alpha}{\lambda + \alpha}$. Thus, we can stochastically dominate $|I_t|$ via coupling with a continuous time process $W_t$ such that $W_0 = |I(0)|$ and $W_0$ evolves according to the following rules:

\begin{enumerate}
    \item For all $t$ such that $|I_t| > 0$, $W_t$ transitions when $|I_t|$ transitions.
    \item For all $t$ such that $|I_t| = 0$, $W_t$ transitions at rate $1$.
    \item When $W_t$ transitions, it transitions to $W_t + 1$ with probability $\frac{\lambda}{\lambda + \alpha}$ and transitions to $W_t - 1$ with probability $\frac{\alpha}{\lambda + \alpha}$.
\end{enumerate}

Let $\tau_0 = \inf\{t:W_t = 0\}$ and observe that $|I_t| \leq W_t$ for all $t \leq \tau_0$. In addition, observe that $W_t$ is a random walk biased toward $0$, and so if $\tau_0 = \inf\{t:W_t = 0\}$, there exists $C > 0$ such that $P(\tau_0 > Cn) \rightarrow 0$ as $n \rightarrow \infty$.
\end{proof}

\subsection{The supercritical regime}\label{subsec:vigilance-supercritical}

Consider the vigilance model with infection rate $\lambda$ and vigilance rate $\alpha$. Define the continuous time process $V_t = |A_t| - \eta |I_t|$. We study the behavior of this process in Lemma \ref{lemma:a_control} and use the results to prove Lemma \ref{lemma:a_time}.

\begin{lemma}\label{lemma:a_control}
Choose $\eta < \epsilon \delta/24$ where $\epsilon$ and $\delta$ satisfy Assumption \ref{assumption:isoperimetric} for $G$ and $\delta \leq 12$. Choose $\lambda > \frac{2 \alpha}{\eta}$ and let $V_t = |A_t| - \eta|I_t|$. Suppose $V_0$ is in the interval $(2 \eta n, 2.1 \eta n)$. Then $|V_t|$ hits $[0,\eta n]$ before it hits $(3\eta n,\infty)$ with probability at least $1 -  \exp \bigg(\frac{-\eta^2 n}{43} \bigg)$. 

\end{lemma}

\begin{proof}
Define the jump process $\{(\jumpS_k, \jumpI_k, \jumpA_k): k\in \mathbb{Z}_{\ge0}\}$ to be the embedded discrete time Markov chain of $(S_t, I_t, A_t)$, which increments whenever an event (infection, isolation or recovery) occurs. Let $p_k$ be the probability that the next event in this process is either an infection or an isolation. Then

$$p_k = \frac{(\alpha + \lambda) (\#\textrm{SI edges in }(\jumpS_k, \jumpI_k, \jumpA_k))}{(\alpha + \lambda)  (\#\textrm{SI edges in }(\jumpS_k, \jumpI_k, \jumpA_k))  + |\jumpA_k| + |\jumpI_k|}.$$
Define
$$U_k = |\jumpA_k| - \eta|\jumpI_k|,$$
and observe that since $\eta < 1/4$ and $\lambda > \frac{2 \alpha}{\eta} > 8 \alpha$, while $U_k > \eta n$

\begin{equation*}
\begin{aligned}
\E(U_{k+1}-U_k|(\jumpS_k, \jumpI_k, \jumpA_k)) &= p_k \bigg(\frac{\alpha}{\alpha+\lambda}(1 + \eta) - \eta \frac{\lambda}{\lambda + \alpha}\bigg) + (1-p_k) \bigg( \frac{-(|\jumpA_k| - \eta|\jumpI_k|)}{|\jumpA_k| + |\jumpI_k|}\bigg)\\
& \leq p_k \bigg(\frac{\alpha}{\lambda}(1 + \eta) - \eta \frac{8\alpha}{\alpha+8\alpha}\bigg) + (1-p
_k)\bigg(\frac{-U_k}{n}\bigg)\\
&\leq p_k \bigg((1/2)\eta(1+\eta) - (8/9)\eta\bigg) + (1-p_k) \bigg( -\eta \bigg)\\
&\leq -\frac{1}{4}\eta.
\end{aligned}
\end{equation*}

Let $\cK = \inf\{k: U_k \le \eta n\}$ be the first time that $U_k$ hits $[0, \eta n]$.
Now define the discrete time process $W_k$ by

$$W_k = U_{k\wedge \cK} + \frac{1}{4}\eta (k\wedge \cK)$$
and note that $W_k$ is a supermartingale. Let time $t=0$ correspond to increment $k=0$ in the embedded discrete process $(\jumpS_k, \jumpI_k, \jumpA_k)$, so that $W_0 = U_{0} = V_0$. Note that for all $k$, $|W_{k+1} - W_k| \leq 2$, so we can apply Azuma's inequality to see that

\begin{equation*}
\prob(W_{k} - W_{0} > .9\eta n) \leq \exp \bigg(\frac{-(.81)\eta^2 n^2}{8k} \bigg).
\end{equation*}
Observe that when $W_{k} - W_{0} \leq .9\eta n$ holds, we have
\begin{equation*}
\begin{aligned}
&U_{k} + \frac{1}{4} \eta k - U_0 \leq .9\eta n\\
&U_{k} \leq U_0 + .9\eta n - \frac{1}{4} \eta k
\end{aligned}
\end{equation*}

Suppose $W_{k} - W_{0} \leq .9\eta n$ holds for $k = 1, \ldots, \mathcal{K}$. This ensures that $U_k$ hits the interval $[0, \eta n]$ before it hits the interval $(3 \eta n, \infty)$. It also ensures we must have $\mathcal{K}-1 \leq 8n$, otherwise

$$U_{\mathcal{K} - 1} \leq U_0 - 2 \eta n + .9\eta n < \eta n$$


which is a contradiction since $\cK = \inf\{k: U_k \le \eta n\}$. Thus we can observe

\begin{equation*}
\begin{aligned}
\prob(\cK > 8n + 1) &\leq P(W_k - W_0 > .9 \eta n \textrm{ for some }k = 1, \ldots, m)\\
&\leq \sum_{k=1}^{\mathcal{K}} \exp \bigg( \frac{- (.81)\eta^2 n^2}{8k}\bigg)\\
&\leq \sum_{k=1}^{8n+1} \exp \bigg( \frac{- (.81)\eta^2 n^2}{8k}\bigg)\\
&\leq (8n+1) \exp \bigg(\frac{-\eta^2 n}{40} \bigg)\\
&\leq \exp \bigg(\frac{-\eta^2 n}{41} \bigg) \textrm{ for large enough }n.
\end{aligned}
\end{equation*}

So long as $U_k$ does not hit the interval $(3 \eta n, \infty)$, $V_t$ does not hit the interval $(3 \eta n, \infty)$, and when $U_{\kappa}$ hits the interval $(0, \eta n]$, $V_{\tau}$ hits the interval $(0, \eta n]$. Thus for large enough $n$, with probability at least $1 - \exp \bigg(\frac{-\eta^2 n}{41} \bigg)$, after $V_t$ hits the interval $(2 \eta n, 2.1 \eta n)$, it enters the interval $(0, \eta n]$ and does so before visiting $(3 \eta n, \infty)$.

\end{proof}

\begin{lemma}\label{lemma:a_time}
Chose $\eta < \epsilon \delta/24$ where $\epsilon$ and $\delta$ satisfy Assumption \ref{assumption:isoperimetric} for $G$ and $\delta \leq 12$. Choose $\lambda > \frac{2 \alpha}{\eta}$ and let $V_t = |A_t| - \eta|I_t|$. Suppose $V_0$ is in the interval $(2 \eta n, 2.1 \eta n)$. Let $\tau = \inf\{t>0 : V_t> 3\eta n\}$. Then there is $c>0$ (depending on $\lambda,\alpha, \eta$) such that $\prob(\tau> e^{c n})\ge 1-e^{-cn}$.

\end{lemma}

\begin{proof}


Recall that $V_0$ is in the interval $(2 \eta n, 2.1 \eta n)$ and let $\tau_1$ be the first time $V_t$ leaves ($\eta n, 3 \eta n)$. We can observe $\tau_1$ is bounded below by the time it takes for $(.9 \eta n)/(1+\eta)\ge \eta n/2$ transitions that change $V_t$ to occur, since any transition can change the value of $V_t$ by at most $1+\eta$. The possible transitions are infections, isolations, recoveries of infected vertices, and recoveries of isolated vertices, and these happen at combined rate

\begin{equation}
    \begin{aligned}
        &(\lambda + \alpha) (\#\textrm{number SI edges}_t) + |A_t| + |I_t|\\
        & \leq (\lambda + \alpha){n \choose 2} + 2n\\
        & \leq (\lambda + \alpha)n^2 \textrm{ for large enough }n
    \end{aligned}
\end{equation}

Therefore the amount of time for the continuous time process $V_t$ to hit $(0, \eta n]$ starting from $(2 \eta n, 2.1 \eta n)$ stochastically dominates the random variable $X \sim \textrm{Gamma}(\eta n/2, (\lambda+\alpha) n^2)$. Thus using  a Chernoff bound

$$\prob\bigg(\tau_1 < \frac{\eta/2}{(\lambda+\alpha) n^2}\bigg) \leq \prob\bigg(X < \frac{\eta/2}{(\lambda+\alpha) n^2}\bigg) \leq \exp(-n).$$

Now consider cycles of $V_t$ hitting the interval $(2 \eta n, 2.1 \eta n)$ and then hitting either $(0, \eta n]$ or $(3 \eta n, \infty)$. Say that a cycle fails if either after hitting the interval $(2 \eta n, 2.1 \eta n)$ $V_t$ hits $(3 \eta n, \infty)$ before it hits $(0, \eta n]$ or $V_t$ leaves $(\eta n, 3 \eta n)$ in less time than $\eta n/2$. Using a union bound, the probability of failure on a cycle is at most
$$\exp \bigg(\frac{-\eta^2 n}{41} \bigg) + \exp(-n)$$
and thus the probability of any failures in $\exp \bigg(\frac{\eta^2 n}{84} \bigg)$ cycles is at most

\begin{equation}
\begin{aligned}
&\exp\bigg(\frac{\eta^2 n}{84}\bigg)\bigg(\exp \bigg(\frac{-\eta^2 n}{41} \bigg) + \exp(-n)\bigg)\\
& \leq \exp\bigg(\frac{-\eta^2 n}{84}\bigg) \textrm{ for large enough }n
\end{aligned}
\end{equation}

Therefore, if we observe no failures in $\exp \bigg(\frac{\eta^2 n}{84} \bigg)$ cycles we will have that if $V_t$ begins below $2.1 \eta n$, it will remain below $3 \eta n$ for at least time

\begin{equation}
    \begin{aligned}
        &\frac{\eta/2}{(\lambda+\alpha) n^2}\exp \bigg(\frac{\eta^2 n}{84}\bigg)\\
        & \geq \exp\bigg(\frac{\eta^2 n}{85}\bigg) \textrm{ for large enough }n
    \end{aligned}
\end{equation}
and this happens with probability at least $1 -\exp\bigg(\frac{-\eta^2 n}{84}\bigg)$.
\end{proof}

We now prove our main result.

\begin{lemma}\label{lemma:renewal}
   Suppose we consider the vigilance model on a graph $G = (V,E)$ where $|V| = n$ satisfying Assumption \ref{assumption:isoperimetric}. Fix $\alpha > 0$ and let  $T_k = \inf \{t: |I_t| = k\}$. Then there exist $\epsilon$ and $\eta$ such that 
 if $|I_0| = (2/3)\epsilon n$, $|A_0| < 2 \eta n$, and $\lambda > \lambda_0$ where $\lambda_0$ depends on $\alpha$, $\epsilon$, and $\eta$, then $\prob(T_{\epsilon n} < T_{(1/3)\epsilon  n}) \geq 1 - \exp \bigg(\frac{-\epsilon \eta^2 n}{85} \bigg)$
\end{lemma}

\begin{proof}

    Start by fixing $\alpha > 0$, choosing $\eta < \epsilon \delta/24$ where $\epsilon$ and $\delta$ satisfy Assumption \ref{assumption:isoperimetric} for $G$, $\delta \leq 12$, and $\lambda_0 = \max \{\frac{2 \alpha}{\eta}, \alpha + 9/\delta\}$.

     From Lemma \ref{lemma:a_time} we have that with probability at least $1 -\exp\bigg(\frac{-\eta^2 n}{84}\bigg)$, $V_t = |A_t| - \eta |I_t| \leq 3 \eta n$ for all $t \leq \exp\bigg(\frac{\eta^2 n}{85}\bigg)$. Thus for $t < \min\{\exp\bigg(\frac{\eta^2 n}{85}\bigg), T_{(1/3)\epsilon n}, T_{\epsilon n}\}$, we have $|I_t| \in ((1/3)\epsilon n,  \epsilon n)$ and $V_t \leq 3\eta n$, which together imply 
     \begin{equation}
         \begin{aligned}
             |A_t| &\leq (3 + \epsilon) \eta n\\
             &\leq 4 \eta n.
         \end{aligned}
     \end{equation}
     From our initial condition on $|I_0|$, the definitions of $T_{(1/3)\epsilon n}$ and $T_{\epsilon n}$, and Assumption \ref{assumption:isoperimetric}, we can observe that $|\partial I_t| \geq \delta (\epsilon/3)n$ for all $t < \min\{\exp\bigg(\frac{\eta^2 n}{85}\bigg), T_{(1/3)\epsilon n}, T_{\epsilon n}\}$. We can further observe that since $|A_t| \leq 4\eta n$,

     \begin{equation}
         \begin{aligned}
             |\partial I_t \cap S_t| &\geq \delta (\epsilon/3) n - 4 \eta n\\
             &\geq \delta \epsilon n/6
         \end{aligned}
     \end{equation}

    by our choice of $\eta$.

    Now consider events that change $|I_t|$. $|I_t|$ decreases by $1$ when either an infected vertex recovers or becomes isolated, and $I_t$ increases by $1$ when a susceptible vertex becomes infected. Thus, for $t < \min\{\exp\bigg(\frac{\eta^2 n}{85}\bigg), T_{(1/3)\epsilon n}, T_{\epsilon n}\}$ the probability that the next event that changes $|I_t|$ increases it by $1$ rather than decreases it by $1$ is

    \begin{equation}
    \begin{aligned}
        &\frac{(\lambda - \alpha)|\partial I_t \cap S_t|}{(\lambda - \alpha)|\partial I_t \cap S_t| + |I_t|}\\
        & \geq \frac{(\lambda - \alpha)\epsilon \delta n/6}{(\lambda - \alpha)\epsilon \delta n/6 + \epsilon n}\\
        & \geq 3/5
    \end{aligned}
    \end{equation}
    by our choice of $\lambda > \lambda_0 \geq \alpha + 9/\delta$.

    Recall the embedded discrete time process $(\jumpS_k, \jumpI_k, \jumpA_k)$. For all increments $k$ that occur before time $\min\{\exp\bigg(\frac{\eta^2 n}{85}\bigg), T_{(1/3)\epsilon n}, T_{\epsilon n}\}$, $|\jumpI_k|$ stochastically dominates a random walk $R_k$ that starts at $R_0 = (2/3)\epsilon n$, increases by $1$ with probability $3/5$ and decreases by $1$ with probability $2/5$ each increment. We continue this coupling until one of the following occurs:
    \begin{enumerate}
        \item $|\jumpI_k|$ hits $\epsilon n$
        \item $|\jumpI_k|$ hits $(1/3)\epsilon n$
        \item $k$ hits $K$ such that increment $K$ occurs after time $\frac{18n}{(\alpha + \lambda)\delta}$
    \end{enumerate}
    after which we let $R_k$ continue independently of $|\jumpI_k|$. We can observe that if all of the following hold
    \begin{enumerate}
        \item $R_{2 \epsilon n} \geq \epsilon n$
        \item $R_k \textrm{ hits }\epsilon n \textrm{ before it hits }(1/3)\epsilon n$
        \item $K > 2\epsilon n$
    \end{enumerate}
    then our coupling ensures that $|\jumpI_k|$ hits $\epsilon n$ before it hits $(1/3)\epsilon$ and this happens for some increment $k < \exp\bigg(\frac{\eta^2 n}{85}\bigg)$. This in turn implies that $T_{\epsilon n} \leq \min\{\exp\bigg(\frac{\eta^2 n}{85}\bigg), T_{(1/3)}\epsilon n\} \leq T_{(1/3)\epsilon n}$. We will bound each of failure of each of these events.

    \begin{enumerate}
    \item Applying a Binomial Chernoff bound, we observe that

    $$\prob(R_{2 \epsilon n} < \epsilon n) < \exp\bigg( \frac{-\epsilon^2 n^2}{225}\bigg).$$

    \item The probability that $R_k$ hits $(1/3)\epsilon n$ before it hits $\epsilon n$ is given by a standard Gambler's ruin analysis:

    $$\prob(R_k \textrm{ hits }(1/3)\epsilon n \textrm{ before it hits }\epsilon n) = \frac{1 - (3/2)^{(1/3)\epsilon n}}{1 - (3/2)^{(2/3)\epsilon n}} \leq \exp\bigg(\frac{-\epsilon n}{9}\bigg).$$

    \item Observe that $|\jumpI_k|$ changes when either an infection occurs, an isolation occurs, or an infected vertext recovers. This happens at combined rate 

    \begin{equation}
        \begin{aligned}
            &(\alpha + \lambda)(\textrm{\# SI edges}) + |\jumpI_k|\\
            &\geq (\alpha + \lambda) (\delta \epsilon n/6).
        \end{aligned}
    \end{equation}
    Thus the number of increments $k$ that occur in time $\frac{18n}{(\alpha + \lambda)\delta}$ stochastically dominates a random variable $Y \sim \textrm{Poisson}(3 \epsilon n)$. Applying a Chernoff bound for $Y$

    \begin{equation}
        \begin{aligned}
            \prob(K < 2 \epsilon n) \leq \prob(Y < 2 \epsilon n) \leq \exp\bigg(\frac{-\epsilon n}{4}\bigg)
        \end{aligned}
    \end{equation}
    \end{enumerate}

    Therefore we can bound the probability of failure by

    \begin{equation}\label{eq:fail}
        \begin{aligned}
                \prob(T_{(1/3)\epsilon n} < T_{\epsilon n}) \leq \prob(R_{2 \epsilon n} < \epsilon n) &+ \prob(R_k \textrm{ hits }(1/3)\epsilon n \textrm{ before it hits }\epsilon n)\\
                &+ \prob(K < 2 \epsilon n) + \prob(\textrm{Lemma \ref{lemma:a_time} fails})   
        \end{aligned}
    \end{equation}
    We have already established $P(R_{2 \epsilon n} < \epsilon n) < \exp\bigg( \frac{-\epsilon^2 n^2}{225}\bigg)$ and $\prob(R_k \textrm{ hits }(1/3)\epsilon n \textrm{ before it hits }\epsilon n) \leq \exp\bigg(\frac{-\epsilon n}{9}\bigg)$ and $\prob(K < 2 \epsilon n) < \frac{-\epsilon n}{4}$ and $P(\textrm{Lemma \ref{lemma:a_time} fails}) < \exp\bigg(\frac{-\eta^2 n}{84}\bigg)$.

Substituting these into equation \ref{eq:fail}, we have 

\begin{equation}
    \begin{aligned}
         \prob(T_{\epsilon n} < T_{(1/3)\epsilon n}) &\geq 1 - \bigg(\exp\bigg( \frac{-\epsilon^2 n^2}{225}\bigg) +  \exp\bigg(\frac{-\epsilon n}{9}\bigg) + \exp\bigg(\frac{-\epsilon n}{4}\bigg) + \exp\bigg(\frac{-\eta^2 n}{84}\bigg)\bigg)\\
         &\geq 1 - \exp \bigg( \frac{-\epsilon \eta^2 n}{85}\bigg)
    \end{aligned}
\end{equation}
for large enough $n$.
\end{proof}

We are now ready to prove Theorem \ref{thm:vigilance-survival}.

\begin{proof}
    Choose some $\epsilon > 0$ and $\delta \in (0,12)$ for which $G$ satisfies Assumption \ref{assumption:isoperimetric}, and fix $\alpha > 0$. Start from all vertices infected so $|I_0| = n$, let $\tau_1 = \inf\{t: |I_t| = (2/3)\epsilon n\}$, and let $T_k^1 = \inf\{t \geq \tau_1: |I_t| = k\}$. For $i = 2,3,\ldots$, let $\tau_i = \inf\{t > \tau_{i-1}:|I_t| = (2/3)\epsilon n\}$ and let $T_k^i = \inf\{t \geq \tau_i: |I_t| = k\}$. By Lemma \ref{lemma:renewal}, if we choose $\eta < \epsilon \delta/24$ and $\lambda > \lambda_0 = \max(\frac{2\alpha}{\eta}, \alpha + 9/\delta)$ then 
    
    $$
    \prob(T_{\epsilon n}^1 < T_{(1/3)\epsilon n}^1) \geq 1 - \exp\bigg(\frac{-\epsilon\eta^2n}{86}\bigg).
    $$

    We can also observe for each $i = 1,2,\ldots$ that if $T_{\epsilon n}^i > T_{(1/3)\epsilon n}^i$ this guarantees $\tau_{i+1} < \infty$ because $\lim_{t\rightarrow\infty} |I_t| = 0$ and to reach $0$ from $\epsilon n$, $|I_t|$ must pass through $(2/3)\epsilon n$. We can thus conclude

    $$
    \{T_{\epsilon n}^j < \infty\} \supseteq \{T_{\epsilon n}^j < T_{(1/3)\epsilon n}^j\} \supseteq \{T_{\epsilon n}^j < T_{(1/3)\epsilon n}^j|T_{\epsilon n}^{j-1} < T_{(1/3)\epsilon n}^{j-1}\} \cap \ldots \cap \{T_{\epsilon n}^1 < T_{(1/3)\epsilon n}^1\}
    $$
 
     and use the strong Markov property to apply Lemma \ref{lemma:renewal} to see that for each $i = 1, 2, \ldots$

     $$
     \prob(T_{\epsilon n}^i > T_{(1/3)\epsilon n}^i|T_{\epsilon n}^{i-1} > T_{(1/3)\epsilon n}^{i-1}) \geq 1 - \exp\bigg(\frac{-\epsilon\eta^2n}{86}\bigg).
     $$

     Taking a union bound on the probability of failure of Lemma \ref{lemma:renewal} for $i = 1,\ldots, j$, we see that

     $$
     \prob(T_{\epsilon n}^j < \infty) \geq \prob(T_{\epsilon n}^j > T_{(1/3)\epsilon n}^j) \geq 1 - j \exp\bigg(\frac{-\epsilon\eta^2n}{86}\bigg).
     $$

     If we choose $j = \exp\bigg(\frac{\epsilon\eta^2n}{172}\bigg)$ then this probability is $1 - \exp\bigg(\frac{-\epsilon\eta^2n}{172}\bigg)$, so with high probability we will have $T_{\epsilon n}^j < \infty$ for $j = \exp\bigg(\frac{\epsilon\eta^2n}{172}\bigg)$.

     To obtain a lower bound on $T_{\epsilon n}^j$, consider that for each $i = 1, \ldots, j$ there must be at least one event that involves a vertex changing states between $T_{\epsilon n}^{i-1}$ and $T_{\epsilon n}^i$. Given our graph $G = (V,E)$, the total rate at which any such event occurs is bounded above by $(\lambda + \alpha)|E| + |V|$. Since $|V| = n$ and $|E| < n^2$, the total rate is bounded above by $(\lambda + \alpha + 1)n^2$. Thus, if for notational convenience we denote $T_{\epsilon n}^0 = 0$ and consider the sum

     $$
     S_j = \sum_{i=1}^j T_{\epsilon n}^i - T_{\epsilon n}^{i-1},
     $$

     $S_j$ dominates the random variable $X \sim \textrm{Gamma}(j, (\lambda + \alpha + 1)n^2)$. We can apply Chebyshev's inequality to see

     \begin{equation}
         \begin{aligned}
             \prob\bigg(S_j < \frac{j}{2(\lambda + \alpha + 1)n^2}\bigg) &\leq \prob\bigg(X < \frac{j}{2(\lambda + \alpha + 1)n^2}\bigg)\\
             &\leq \prob\bigg(\lvert X - \frac{j}{(\lambda + \alpha + 1)n^2} \rvert  >\frac{j}{2(\lambda + \alpha + 1)n^2}\bigg)\\
             &\leq \frac{j \cdot 4(\lambda + \alpha + 1)^2n^4}{j^2 (\lambda + \alpha + 1)^2n^4}\\
             &= \frac{4}{j}.
         \end{aligned}
     \end{equation}

    On the event $\{T_{\epsilon n}^j < \infty\} \cap \{S_j \geq \frac{j}{2(\lambda + \alpha + 1)n^2}\}$ we have that $\tau \geq \frac{j}{2(\lambda + \alpha + 1)n^2}$. Settting $j =  \exp\bigg(\frac{\epsilon\eta^2n}{172}\bigg)$ and taking a union bound on the probability of failure for large $n$ we observe

    \begin{equation}
        \begin{aligned}
            \prob\bigg(\tau <  \exp\bigg(\frac{\epsilon\eta^2n}{344}\bigg)\bigg) &\leq \prob\bigg(\tau <  \frac{\exp\bigg(\frac{\epsilon\eta^2n}{172}\bigg)}{2(\lambda + \alpha + 1)n^2}\bigg)\\
            &\leq \prob(T_{\epsilon n}^j = \infty) + \prob\bigg(S_j < \frac{\exp\bigg(\frac{\epsilon\eta^2n}{172}\bigg)}{2(\lambda + \alpha + 1)n^2}\bigg)\\
            &\leq \prob(T_{\epsilon n}^j = \infty) + \prob\bigg(X < \frac{\exp\bigg(\frac{\epsilon\eta^2n}{172}\bigg)}{2(\lambda + \alpha + 1)n^2}\bigg)\\
            &\leq \exp\bigg(\frac{-\epsilon\eta^2n}{172}\bigg) + 4\exp\bigg(\frac{-\epsilon\eta^2n}{172}\bigg)\\
            &\rightarrow 0 \textrm{ as }n \rightarrow \infty,
        \end{aligned}
    \end{equation}

     and this completes the proof.
\end{proof}

\bibliographystyle{abbrv}
\bibliography{ref, references}

\newpage
\appendix

\end{document}